\documentclass{amsart}
\usepackage{amsmath}
\usepackage{dsfont}
  \usepackage{paralist}
  \usepackage{graphics} 
  \usepackage{epsfig} 
\usepackage{graphicx}  \usepackage{epstopdf}
 \usepackage[colorlinks=true]{hyperref}
 \usepackage{amsmath}
\usepackage{amsfonts}
\usepackage{amssymb}
\usepackage{ mathrsfs }
\usepackage{amsthm}
\usepackage[pagewise]{lineno}
\usepackage{enumitem}
\usepackage{graphicx,color,xcolor,tikz}
\usepackage{bm}
\usepackage{hyperref}
\hypersetup{
    colorlinks=true,                         
    linkcolor=blue, 
    citecolor=red, 
  } 
\hypersetup{urlcolor=blue, citecolor=red}

\newtheorem{thm}{Theorem}

\newtheorem{lem}[thm]{Lemma}

\newtheorem{defi}[thm]{Definition}
\newtheorem{rem}[thm]{Remark}

\def\I{\mathcal I}
\def\K{\mathcal K}

\def\P{\mathcal P}

\def\R{\mathbb R}
\def\V{\mathbb V}

\def\e{\varepsilon}

\newcommand{\pt}{\partial}

\newcommand{\abs}[1]{\ensuremath{\left|#1\right|}}
\newcommand{\card}[1]{\left\lvert#1\right\rvert}
\newcommand{\norm}[1]{\ensuremath{\left\|#1\right\|}}
\newcommand{\ov}[1]{\ensuremath{\overline{#1}}}
\newcommand{\ovo}[1]{\ensuremath{\overline{\overline{#1}}}}
\makeatletter
\newcommand{\doublewidetilde}[1]{{%
  \mathpalette\double@widetilde{#1}%
}}
\newcommand{\double@widetilde}[2]{%
  \sbox\z@{$\m@th#1\widetilde{#2}$}%
  \ht\z@=.9\ht\z@
  \widetilde{\box\z@}%
}
\makeatother

\title[Part 1- Modeling and Well-posedness]{Microscopic tridomain model of electrical activity in the heart with dynamical gap junctions. Part 1- Modeling and Well-posedness}

\subjclass{65N55
, 35A01
, 35A02
, 65M
, 92C30 
}
 \keywords{Tridomain model, Global existence, Uniqueness, Weak solution, Gap junctions, Cardiac electro-physiology.}

\author{Fakhrielddine Bader$^*$ }
\address[Fakhrielddine Bader]{Laboratoire Mathématiques, Image et Applications (MIA), La Rochelle Université, Avenue Michel Crépeau, La Rochelle, France}
\email{fakhrielddine.bader@univ-lr.fr}

\author{Mostafa Bendahmane}
\address[Mostafa Bendahmane]{Institut de Mathématiques de Bordeaux and INRIA-Carmen Bordeaux Sud-Ouest, Université de Bordeaux, 33076 Bordeaux Cedex, France}
\email{mostafa.bendahmane@u-bordeaux.fr}

\author{Mazen Saad}
\address[Mazen Saad]{Laboratoire de Mathématiques Jean Leray (LMJL), École Centrale de Nantes, UMR 6629 CNRS, 1 rue de Noë, Nantes, France}
\email{mazen.saad@ec-nantes.fr}

\author{Raafat Talhouk}
\address[Raafat Talhouk]{Mathematics Laboratory, Doctoral school of Sciences and Technologies, Lebanese University, Hadat Beirut, Lebanon}
\email{rtalhouk@ul.edu.lb}

\thanks{$^*$ Corresponding author: fakhri.bader.fb@gmail.com}

\begin{document}
\maketitle



\begin{abstract}
We present a novel microscopic tridomain model describing the electrical activity in cardiac tissue with dynamical gap junctions.  The microscopic tridomain system consists of three PDEs modeling the tissue electrical conduction in the intra- and extra-cellular domains,  supplemented by a nonlinear ODE system for the dynamics of the ion channels and the gap junctions. We establish the global existence and uniqueness of the weak solutions to our microscopic tridomain model.  The global existence of solution, which constitutes the main result of this paper, is proved by means of an approximate non-degenerate system, the Faedo-Galerkin method, and an appropriate compactness argument.  
\end{abstract}
\tableofcontents

\section{Introduction}

The heart study started since more than two millennia back. This organ, about the size of its owner's clenched fist, contracts rhythmically to circulate blood throughout the body, while other organs like the brain and lungs, were thought to exist to cool the blood. Until this day the heart
keeps the position of one of the most important and the most studied organs in the human body. Especially, cardiovascular disease (CVD) leading to heart attack, is the top cause of death in the worldwide as announced by the "World Health Organization" in 2019. Given the large number of related pathologies, there is an important need for understanding the chemical and electrical phenomena taking place in the cardiac tissue.


In fact, the heart is a muscular organ can be viewed as double pump consisting of four chambers: upper left and right atria, and lower left and right ventricles. These four chambers are surrounded by a cardiac tissue that is organized into muscle fibers. These fibers form a network of cardiac muscle cells called "cardiomyocytes" connected end-to-end by junctions called intercalated discs. Intercalated discs contain gap junctions and desmosomes. Gap junctions transverse of contiguous cells and connect the cytoplasm of one cell to the cytoplasm of the adjacent cell. Cardiac tissue use gap junctions to spread action potential to nearby cells. This allows the heart to generate a single continuous and forceful contraction that pumps the blood throughout the body \cite{keener,katz10}.

 The structure of cardiac tissue (myocarde) studied in this paper is characterized at two different scales (see Figure \ref{cardio}). At microscopic scale, the cardiac tissue consists of two intracellular media which contains the contents of the cardiomyocytes (the cytoplasm) that are connected by gap junctions and the other is called extracellular and consists of the fluid outside the cardiomyocytes cells. Each intracellular medium and the extracellular one are separated by a cellular membrane (the sarcolemma). While at the macroscopic scale, this domain is well considered as a single domain (homogeneous).
 \begin{figure}[h!]
  \centering
  \includegraphics[width=13cm]{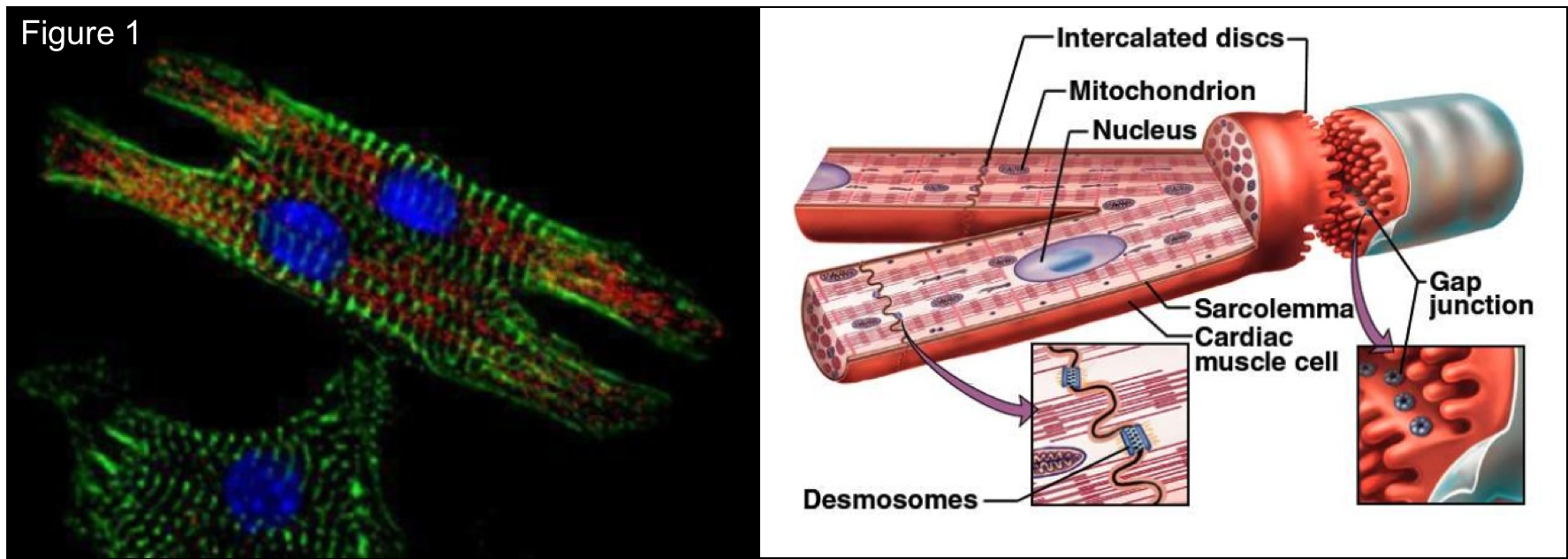}
  \caption{Representation of the cardiomyocyte structure}
  \url{http://www.cardio-research.com/cardiomyocytes}
  \label{cardio}
 \end{figure}
 
 It should be noted that there is a difference between the chemical composition of the cytoplasm and that of the extracellular medium. This difference plays a very important role in cardiac activity. On the one hand, the sarcolemma allows the penetration of inorganic ions (sodium, potassium, calcium,...) and proteins, some of which play a passive role and others play an active role powered by cellular metabolism. In particular, the concentration of anions (negative ions) in cardiomyocytes is higher than in the external environment. This difference of concentrations creates a transmembrane potential, which is the difference in potential at the sarcolemma between each intracellular medium and the extracellular one. On the other hand, gap junctions allows the movement of not only inorganic ions but also organic ions between two adjacent cells \cite{keener}. It provide the pathways for intracellular current flow, enabling coordinated action potential propagation. So, the difference of chemical through the gap junction  creates a gap potential, which is the difference in potential between these two  intracellular media.  Our  model that describes the electrical activity in the cardiac tissue including the gap junctions, is called by "tridomain model". 
From the mathematical viewpoint, the microscopic tridomain model consists of three quasi-static equations, two for the electrical potential in the intracellular medium and one for the extracellular medium, coupled by a nonlinear ODE system  at each membrane (the sarcolemma) and by a linear one at gap junction for the dynamics of the ion channels. These equations depend on scaling parameter $\e$ whose is the ratio between the microscopic scale and the macroscopic one. The microscopic tridomain model was proposed three years ago \cite{tveito17,tveito19} in the case of just two coupled cells compared to our model which is defined on larger collections of cells.


  The goal of the present paper is to investigate existence and uniqueness of solutions of the tridomain equations, coupled with an ionic model, namely the FitzHugh–Nagumo model. We mention some works in the literature on the bidomain model that gives a macroscopic description of the cardiac tissue from  two inter-penetrating domains which are the intracellular and extracellular domains at the microscopic scale.  The first mathematical formulation of this model was constructed by Tung \cite{tung}. This variant leads to two quasi-static whose unknowns are intra- and extracellular electric potentials coupled with non linear ordinary differential equations called ionic models at the membrane. Next, Krassowska and Neu \cite{neukra} have proposed to represent cells by large cylinders connected to each other by narrow channels and then have applied the two-scale asymptotic method to formally obtain the bidomain model from the microscopic problem. In particular, they are considered that these narrow channels precisely model gap-junctions ("low-resistance connections between cells"). There are some references dealing with the well-posedness of this model. First, global existence in time and uniqueness for the solution of the micro- and macroscopic bidomain model  coupled with FitzHugh–Nagumo simplification for the ionic currents, is proven in \cite{colli02,colli05}. It is based on a reformulation of the bidomain problem as a Cauchy system for an evolution variational inequality in a properly chosen Sobolev space. Next, the authors in \cite{marco} used Schauder's fixed point theorem to establish the well-posedness of the macroscopic bidomain problem with a generalized phase-I Luo–Rudy ionic model \cite{luo2}. The authors in \cite{boulakia} have studied the well-posedness of the macroscopic bidomain model coupled to a third PDE that describes the electrical potential of the surrounding tissue within the torso. The existence of a global solution of the latter model is proved using the Faedo-Galerkin method for a wide class of ionic models (including Mitchell-Schaeffer model \cite{mitchell}, FitzHugh-Nagumo \cite{fitz,nagumo}, Aliev-Panfilov \cite{aliev}, and Roger-McCulloch \cite{roger}). Furthermore, in \cite{char09}, existence of a global solution of the macroscopic bidomain model is proved only for the last three ionic models, using a semi-group approach and the Galerkin technique. While uniqueness, however, is achieved only for the FitzHugh–Nagumo ionic model in the two previous works. Moreover, the authors in \cite{bend} proved the existence and uniqueness of solution of the macroscopic bidomain model by using the Faedo-Galerkin method (see for instance \cite{bendunf19} where the authors prove the well-posedness of solution for the microscopic bidomain model using the same technique). In the present work, we prove the existence of solution for the novel microscopic tridomain model by a constructive method based on Faedo-Galerkin approach without the restrictive assumption, usually found in the literature, on the conductivity matrices to have the same basis of eigenvectors or to be diagonal matrices (see for instance \cite{char09}  where the authors prove the existence of a local in time strong solution of the bidomain equations). It is worth to mention that our approach is innovative and cannot be found in the literature in the context of existence of solutions to the microscopic tridomain model.

  \textit{The main contribution of the present paper.} The cardiac tissue structure studied at micro-macro scales. We start by modeling the microscopic tridomain model by taking account the presence of gap junctions as connection between adjacent cardiac cells. Next, we formulate our tridomain model in dimensionless form with the hope to get more insight in the meaning of the microscopic and macroscopic scales. Finally, we end by proving the well-posedness of the microscopic tridomain problem by using Faedo-Galerkin method, a priori estimates and $L^2$-compactness argument on the membrane surface.  
  
  \textit{The outline of the paper is as follows.} In Section \ref{geotrid}, we describe the geometry of cardiac tissue in the presence of gap junction and some notations and explanations on the boundary conditions are introduced. Furthermore, we introduce in detail our microscopic tridomain model in the cardiac tissue structure. In Section \ref{main_results_trid}, our main result is stated: existence and uniqueness of a weak solutions. In Section \ref{Exist_sol_gap}, we shall completely define and prove existence, uniqueness of a weak solutions. It is based on  Faedo-Galerkin technique, a priori estimates and compactness results. The results are obtained under minimal regularity assumptions on the data.

\section{Tridomain modeling of the cardiac tissue}\label{geotrid}
The aim of this section is to describe the geometry of cardiac tissue and to present the microscopic tridomain model of the heart.
\subsection{Geometry of heart tissue}
The cardiac tissue $\Omega \subset \R^d \ (d\geq 3)$ is considered as a heterogeneous periodic domain with a Lipschitz boundary $\pt \Omega$. The structure of the tissue is periodic at microscopic scale related to small parameter $\e$, see Figure \ref{two_scale_gap}. 

 \begin{figure}[h!]
  \centering
  \includegraphics[width=12cm]{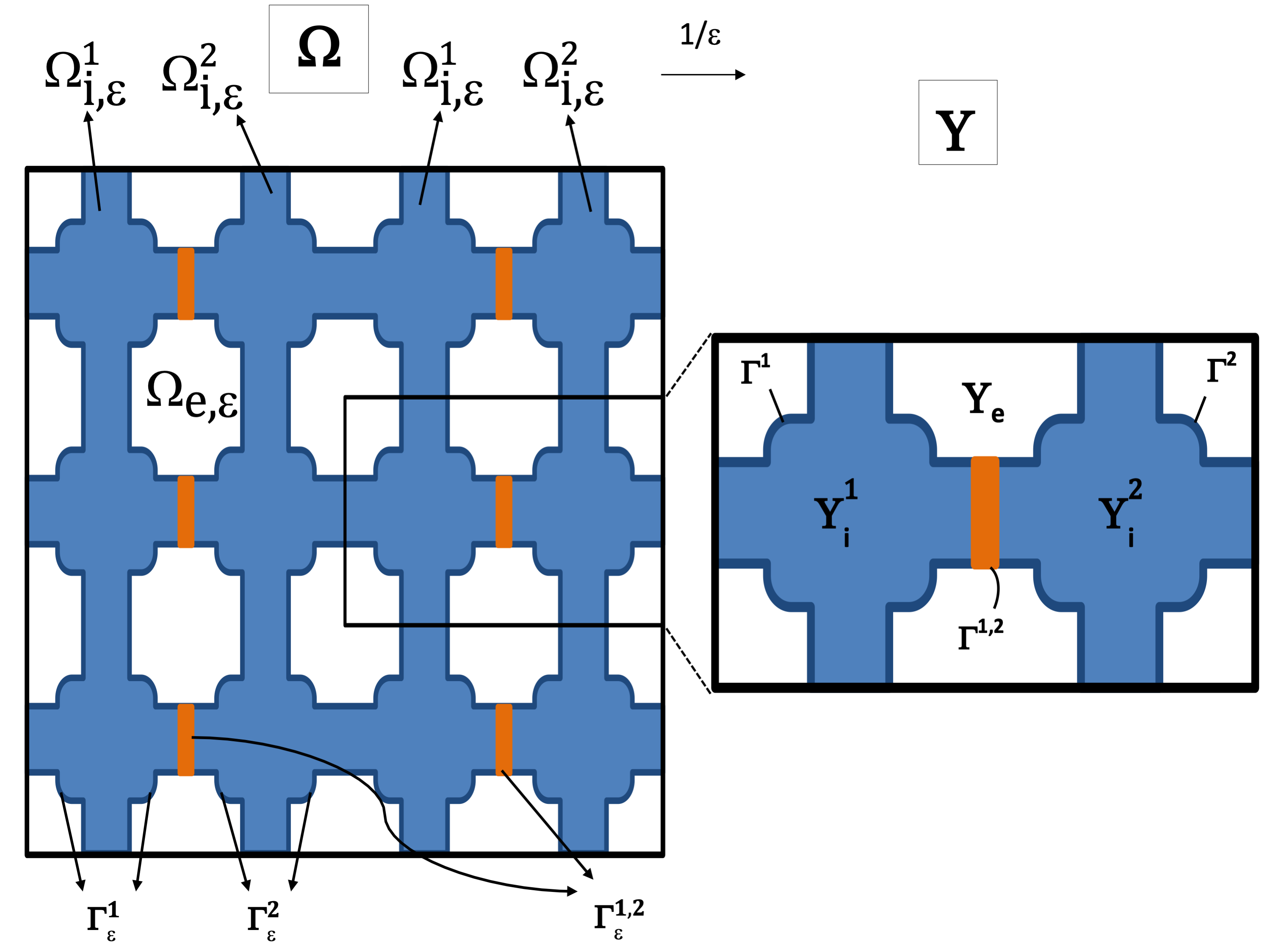}
  \caption{(Left) Periodic heterogeneous domain $\Omega.$ (Right) Unit cell $Y$ at $\e$-structural level.}
  \label{two_scale_gap}
 \end{figure} 

 Following the standard approach of the homogenization theory, this structure is featured by $\ell^\text{mic}$ characterizing the microscopic length of a cell. Under the one-level scaling, the characteristic length $\ell^\text{mic}$ is related to a given macroscopic  length $L$ (of the cardiac fibers), such that the scaling parameter $\e$ introduced by:
 $$\e=\frac{\ell^\text{mic}}{L}.$$

 Physiologically, the cardiac cells are connected by many gap junctions. Therefore, geometrically, the domain $\Omega$ consists of two intracellular media $\Omega_{i,\e}^{k}$ for $k=1,2,$ that are connected by gap junctions $\Gamma^{1,2}_{\e}=\pt \Omega_{i,\e}^{1} \cap \pt \Omega_{i,\e}^{2}$ and extracellular medium $\Omega_{e,\e}$ (for more details see \cite{tveito17,tveito19}). 
 Each intracellular medium $\Omega_{i,\e}^{k}$ and the extracellular one  $\Omega_{e,\e}$ are separated by the surface membrane $\Gamma_{\e}^{k}$ (the sarcolemma) which is expressed by: $$\Gamma_{\e}^{k}=\pt \Omega_{i,\e}^{k} \cap \pt \Omega_{e,\e}, \text{ with } k=1,2,$$ while the remaining (exterior) boundary is denoted by $\pt_{\text{ext}} \Omega$. We can observe that the intracellular domains as a perforated domain obtained from $\Omega$ by removing the holes which correspond to the extracellular domain $\Omega_{e,\e}.$ 
 
  We can divide $\Omega$ into $N_{\e}$ small elementary cells $Y_{\e}=\overset{d}{\underset{n=1}{\prod }}]0,\e \, \ell^\text{mic}_n[,$  with $\ell^\text{mic}_1,\dots,\ell^\text{mic}_d$ are positive numbers. These small cells are all equal, thanks to a translation and scaling by $\e,$ to the same unit cell of periodicity called the reference cell  $Y=\overset{d}{\underset{n=1}{\prod }}]0,\ell^\text{mic}_n[.$ So, the $\e$-dilation of the reference cell $Y$ is defined as the following shifted set $Y_{\e,h}:$
\begin{equation}
 Y_{\e,h}:=T^h_\e+\e Y=\lbrace \e \xi : \xi \in h_\ell+Y \rbrace,
\label{trans_Y}
 \end{equation}
 where $T_\e^h$ represents the translation of $\e h$ with $h=( h_1,\dots, h_d ) \in\mathbb{Z}^d$ and  $h_\ell:=( h_1\ell^\text{mic}_1,\dots,  h_d \ell^\text{mic}_d ).$\\
 Therefore, for each macroscopic variable $x$ that belongs to $\Omega,$ we define the corresponding microscopic variable $y\approx\dfrac{x}{\e}$ that belongs to $Y$ with a  translation. Indeed, we have:
 \begin{equation*}
 x \in \Omega \Rightarrow \exists h \in\mathbb{Z}^d  \ \text{ such that }  \ x \in Y^h_\e \Rightarrow x=\e (h_\ell+y) \Rightarrow y=\dfrac{x}{\e}-h_\ell \in Y.
 \end{equation*}

 Since, we will study the behavior of the functions $u(x,y)$ which are $\textbf{y}$-periodic, so by periodicity we have $u\left( x,\dfrac{x}{\e}-h_\ell\right)  =u\left( x,\dfrac{x}{\e}\right) .$ By notation, we say that $y=\dfrac{x}{\e}$ belongs to $Y.$

 We are assuming that the cells are periodically organized as a regular network of interconnected cylinders at the microscale. The microscopic unit cell $Y$ is also divided into three disjoint connected parts: two intracellular parts $Y_i^{k}$ for $k=1,2,$ that are connected by an intercalated disc (gap junction) $\Gamma^{1,2}$ and extracellular part $Y_e.$ Each intracellular parts $Y_i^k$ and the extracellular one are separated by a common boundary $\Gamma^{k}$ for $k=1,2.$ So, we have:
 \begin{equation*}
 Y:=\overline{Y}_i^{1} \cup \overline{Y}_i^{2} \cup \overline{Y}_e, \quad \Gamma^{k}:= \pt Y_i^{k} \cap \pt Y_e,\quad\Gamma^{1,2}:= \pt Y_i^{1} \cap \pt Y_i^{2},
 \end{equation*}
  with $k=1,2.$ 
In a similar way, we can write the corresponding common periodic boundary as follows:
  \begin{equation}
\Gamma_{\e,h}=T^h_\e+\e \Gamma=\lbrace \e \xi : \xi \in h_\ell+\Gamma \rbrace,
\label{trans_gamma}
 \end{equation}
with $T^h_\e$ denote the same previous translation, $\Gamma_{\e,h}:=\Gamma^{k}_{\e,h},\Gamma^{1,2}_{\e,h}$ and $\Gamma:=\Gamma^{k},\Gamma^{1,2}$ for $k=1,2$.

 In summary, the intracellular and extracellular media can be described as follows:  
 \begin{align*}
&\Omega_{i,\e}^{k}=\Omega \cap \underset{h\in \mathbb{Z}^d}{\bigcup} Y^k_{i,\e,h}, \quad \Omega_{e,\e}=\Omega \cap \underset{h\in \mathbb{Z}^d}{\bigcup} Y_{e,\e,h},
\\& \quad \Gamma_{\e}^{k}=\Omega \cap \underset{h\in \mathbb{Z}^d}{\bigcup} \Gamma^{k}_{\e,h} \text{ and } \Gamma_{\e}^{1,2}=\Omega \cap \underset{h\in \mathbb{Z}^d}{\bigcup} \Gamma^{1,2}_{\e,h},
\end{align*}
where $Y^k_{i,\e,h},$ $Y_{e,\e,h}$ and $\Gamma_{\e}^{k},\Gamma_{\e}^{1,2}$ are respectively defined as \eqref{trans_Y}-\eqref{trans_gamma} for $k=1,2$. 

 Both sets $\Omega_{i,\e}^{k},$ $k=1,2$ and $\Omega_{e,\e}$ are assumed to be connected Lipschitz domains so that a Poincaré-Wirtinger inequality is satisfied in both domains. The boundaries $\Gamma^{k},$ $k=1,2$ and $\Gamma^{1,2}$ are smooth manifolds such that $\Gamma^{k}_{\e},$ $k=1,2$ and $\Gamma^{1,2}_{\e}$ are smooth and connected. 
\subsection{Microscopic tridomain model}
A vast literature exists on the bidomain modeling of the heart, we refer to \cite{char,colli02,colli05,colli12} for more details. Here, we define a novel microscopic tridomain model described in detail in \cite{tveito17,tveito19} and used in our investigations, as well the models chosen for the membrane and gap junctions dynamics. In the sequel, the space-time set $(0,T)\times O$ is denoted by $O_T$ in order to simplify the notation. There are a few references dealing with the tridomain model for other cells types, e.g. cardiomyocytes and fibroblasts \cite{sachse} and for simulating bioelectric gastric pacing \cite{sathar15,sathar18}. \\
\paragraph{\textbf{Basic equations.}} The basic tridomain equations modeling the propagation of cardiac action potentials at cellular level in the presence of gap junctions which can be formulated as follows. First, we know that the structure of the cardiac tissue can be viewed as composed by two intracellular spaces  $\Omega_{i}^{k}$ for $k=1,2,$ that are connected by gap junction $\Gamma^{1,2}$ and the extracellular space $\Omega_e.$ The membrane $\Gamma^{k}$ is defined by the intersection between each intracellular domain $\Omega_{i}^{k}$ and the extracellular one with $k=1,2$. 
  
  Thus, the membrane $\Gamma^{k}$ is pierced by proteins whose role is to ensure ionic transport between the two media (intracellular and extracellular) through this membrane. So, this transport creates an electric current.\\ Using Ohm's law, the intracellular electrical potentials $u_{i}^{k}$ and extracellular one $u_{e}$ are respectively related to the current volume densities $J_{i}^{k}$ and $J_{e}$ for $k=1,2$ :
\begin{align*}
&J_{i}^{k}=\mathrm{M}_{i}\nabla u_{i}^{k}, \ \text{in} \ \Omega_{i,T}^{k}:=(0,T)\times\Omega_{i}^{k},
\\& J_{e}=\mathrm{M}_{e}\nabla u_{e}, \ \text{in} \ \Omega_{e,T}:=(0,T)\times\Omega_{e},
\end{align*}
where $\mathrm{M}_{j}$ represents the corresponding conductivity of the tissue for $j=i,e$ (which are assumed to be
isotropic at the microscale and are given in mS/cm).\\
In addition, the \textit{transmembrane} potential $v^{k}$ is known as the  potential at the membrane $\Gamma^{k}$   which is defined as follows:
\begin{equation*}
v^{k}=(u_{i}^{k}-u_{e})\vert_{\Gamma^{k}} : (0,T)\times \Gamma^{k} \mapsto \R  \text{ for } k=1,2.
\end{equation*}
 
  Moreover, we assume the intracellular and extracellular spaces are source-free and thus the intracellular and extracellular potentials are solutions to the elliptic equations:
\begin{equation}
\begin{aligned}
&-\text{div}J_{i}^{k}=0, \ \text{in} \ \Omega_{i,T}^{k},
\\& -\text{div}J_{e}=0, \ \text{in} \ \Omega_{e,T},
\end{aligned}
\label{pbini_gap}
\end{equation}
with $k=1,2.$
 
 According to the  current conservation law, the surface current density $\I_{m}^{k}$ is now introduced:
\begin{equation}
\I_{m}^{k}=-J_{i}^{k}\cdot n_i^{k}=J_{e}\cdot n_{e}, \ \text{on} \ \Gamma^{k}_{T}:=(0,T)\times \Gamma^{k},
\label{cond_bord_sarco} 
\end{equation}
with $n_i^{k}$ denotes is the (outward) normal pointing out from $\Omega_{i,\e}^{k}$ for $k=1,2$ and $n_{e}$ is the normal pointing out from $\Omega_{e,\e}.$

 The membrane has both a capacitive property schematized by a capacitor and a resistive property schematized by a resistor. On the one hand, the capacitive property depends on the formation of the membrane which can be represented by a capacitor of capacitance  $C_m$ (the capacity per unit area of the membrane is given in $\mu$F$/$cm$^2$). We recall that the quantity of the charge of a capacitor is $q^{k}=C_m v^{k}.$ Then, the capacitive current $\I_c^{k}$ for $k=1,2$ is the amount of charge that flows per unit of time:
\begin{equation*}
\I_{c}^{k}=\pt_t q^{k}=C_m\pt_t v^{k}.
\end{equation*}
On the other hand, the resistive property depends on the ionic transport between the intracellular and extracellular media. Then, the resistive current $\I_r$ is defined by the ionic current $\I_{ion}^{k}$ measured from the intracellular to the extracellular medium which depends on the transmembrane potential $v^{k}$ and the gating variable $w^{k} : \Gamma^{k} \mapsto \R$ with $k=1,2$. Moreover, the total transmembrane current $\I_{m}^{k}$ (see \cite{colli12}) is given by:
\begin{equation*}
\I_{m}^{k}=\I_{c}^{k}+\I_{r}^{k}-\I_{app}^{k} \text{ on } \Gamma^{k}_{T},
\end{equation*} 
with $\I_{app}^{k}$ is the applied current of the membrane surface for $k=1,2$ (given in $\mu$A/cm$^{2}$).\\
Consequently, due to the dynamics of the ionic fluxes through the cell membrane, its electrical potential $v^{k}$ satisfies the following dynamic condition on $\Gamma^{k}$ involving the gating variable $w^{k}$:
\begin{equation}
\begin{aligned}
&\I_{m}^{k}=  C_m\pt_t v^{k}+\I_{ion}\left(v^{k},w^{k}\right)-\I_{app}^{k}  &\text{ on } \Gamma^{k}_{T},
\\ &\pt_t w^{k}-H\left(v^{k},w^{k}\right)=0 &\text{ on } \Gamma^{k}_{T}.
\label{cond_dyn_sarco}
\end{aligned}
\end{equation}
Furthermore, the functions $H$ and $\I_{ion}$ correspond to the ionic model of membrane dynamics. All surface current densities $\I_{m}^{k}$ for $k=1,2$ and $\I_{ion}$ are given in $\mu$A/cm$^{2}$. Moreover, time is given in ms and length is given in cm.
 
  In addition, we represent the gap junction between intra-neighboring cells by a passive model. This model includes several state variables in addition to the gap junction potential $s$  which is defined as follows:  
 \begin{equation*}
 s=(u_{i}^{1}-u_{i}^{2})\vert_{\Gamma^{1,2}} : (0,T)\times \Gamma^{1,2} \mapsto \R.
 \end{equation*}
The ionic current $\I_{1,2}$ through the gap junction $\Gamma^{1,2}$ defined by:
\begin{equation}
\I_{1,2}=-J_{i}^{1}\cdot n_i^{1}=J_{i}^{2}\cdot n_i^{2}, \ \text{on} \ \Gamma^{1,2}_{T}:=(0,T)\times \Gamma^{1,2}.
\label{cond_bord_gap} 
\end{equation}
Similarly, the ionic current $\I_{1,2}$ at a gap junction $\Gamma^{1,2}$ represents the sum of the capacitive and resistive currents. Consequently, regarding the dynamic structure of the gap junction, its electrical potential $s$ satisfies the following dynamic condition on $\Gamma^{1,2}:$
\begin{equation}
\begin{aligned}
&\I_{1,2}=  C_{1,2}\pt_t s+\I_{gap}\left(s\right)  &\text{ on } \Gamma^{1,2}_{T},
\label{cond_dyn_gap}
\end{aligned}
\end{equation}
where $C_{1,2}$ represents the capacity per unit area of the intercalated disc and $\I_{gap}$ represents the corresponding resistive current. In general, the value of $C_{1,2}$ is set to $C_m/2$ because the intercalated disc is assumed to be a membrane of thickness twice as large as the cell membrane, and the specific capacitance of a capacitor $C_m$ formed by two parallel plates separated by an insulator may be assumed to be inversely proportional to the thickness of the insulator \cite{tveito19}. \\

 \paragraph{\textbf{Non-dimensional analysis.}}  We use the microscopic model given in the previous part without the parameter scaling $\e.$ In the non-dimensionalization procedure, $\e$ will appear also in each boundary conditions due to the scaling of the involved quantities (see \cite{henri,colli12,BaderDev} for the bidomain case). 

 As a natural assumption for homogenization, we want to formulate the tridomain equations \eqref{pbini_gap}-\eqref{cond_dyn_gap} in dimensionless form with the hope to get more insight in the meaning of the parameter $\e.$  We define the dimensionless parameter $\e$ as the ratio between the microscopic length $\ell^{mic}$ and the macroscopic length $L$, i.e.
 $$\e=\dfrac{\ell^{mic}}{L}.$$ 
 Using all fundamental material constants, several additional time and length constants can be formulated. For convenience, the macroscopic length is defined as $L=\sqrt{ R_m \lambda\ell^{mic}}$, the membrane time constant $\tau_m$ is given by:
$$\tau_m=R_mC_m,$$
where $R_m$ is the resistance of the passive membrane and $\lambda$ is a normalization of the conductivity matrix $\mathrm{M}_j$ for $j=i,e$.\\
After that, we can convert the microscopic tridomain problem into a non-dimensional form by scaling space and time with the constants, such as, $$x=L\widehat{x} \text{ and } t=\tau_m\widehat{t}.$$ 
We take $\widehat{x}$ to be the variable at the macroscale (slow variable),
$$y:=\dfrac{\widehat{x}}{\e}$$
to be the microscopic space variable (fast variable) in the unit cell $Y$. We also scale the electric potentials for $k=1,2$:
\begin{align*}
& u_{i}^{k}=\delta v \ \widehat{u}_{i}^{k}, \ u_{e}=\delta v \ \widehat{u}_{e},  
\\ & \text{ and } w^{k}=\delta w \ \widehat{w}^{k}
\end{align*}
where $\delta v,$ $\delta w$ are respectively the convenient units to measure the electric potentials and the gating variable. Furthermore, we normalize the conductivities matrices as follows
\begin{equation*}
\widehat{\mathrm{M}}_{j}=\dfrac{1}{\lambda} \mathrm{M}_{j}, \text{ for } j=i,e,
\end{equation*}
and we nondimensionalize the ionic functions $\I_{ion},$ $H$, the applied current $\I_{app}^{k},$ $k=1,2,$ and the gap current $\I_{gap}$ by using the following scales:
\begin{align*}
& \widehat{\I}_{ion}\left(\widehat{v}^{k},\widehat{w}^{k} \right)=\dfrac{R_m}{\delta v} \I_{ion}\left(v^{k},w^{k} \right), \quad \widehat{H}\left(\widehat{v}^{k},\widehat{w}^{k} \right)=\dfrac{\tau_m}{\delta w} H\left(v^{k},w^{k} \right), \\ & \widehat{\I}_{app}^{k}=\dfrac{R_m}{\delta v} \I_{app}^{k}, \text{ and } \widehat{\I}_{gap}\left(\widehat{s}\right)=\dfrac{R_{m}C_{m}}{\delta v C_{1,2}} \I_{gap}\left(s \right), 
\end{align*}
where $\widehat{v}^{k}=\widehat{u}_{i}^{k}-\widehat{u}_{e}$ for $k=1,2$ and $\widehat{s}=\widehat{u}_{i}^{1}-\widehat{u}_{i}^{2}.$
\begin{rem}
Recalling that the dimensionless parameter $\e,$ given by $\e:=\sqrt{\dfrac{\ell^{mic}}{R_m \lambda}},$ is the ratio between the microscopic cell length $\ell^{mic}$ and the macroscopic length $L$, i.e. $\e=\ell^{mic}/L$  and solving for $\e,$ we obtain
\begin{equation*}
\e=\dfrac{L}{R_m \lambda}.
\end{equation*}
As pointed in \cite{henri,colli12}, we can consider for a typical cardiac cell that $\ell^{mic}=100\mu$m, $\lambda= 5$mS/cm and $R_m= 10\,000$ $\Omega$cm$^{2},$ leading to $\e=7.1\times 10^{-3}.$ 
\end{rem}
\begin{rem} Using all scaling parameters, we obtain the dimensionless of gap boundary condition \eqref{cond_dyn_gap} as follows  
\begin{equation*}
 \e\dfrac{C_{1,2}}{C_m}\left(\pt_{\widehat{t}} \widehat{s}+\widehat{\I}_{gap}\left(\widehat{s}\right) \right) =\widehat{\I}_{1,2}  \ \text{ on } \ \Gamma_{\e,T}^{1,2}.
\end{equation*}
As previously stated, we can consider $C_{1,2}=C_{m}/2$ so we rewrite the above equation as follows
\begin{equation*}
\frac{\e}{2}\left(\pt_{\widehat{t}} \widehat{s}+\widehat{\I}_{gap}\left(\widehat{s}\right) \right) =\widehat{\I}_{1,2}  \ \text{ on } \ \Gamma_{\e,T}^{1,2}.
\end{equation*} 
\end{rem}

Cardiac tissue exhibits a number of significant inhomogeneities in particular those related to cell-to-cell communications. Rescaling the equations \eqref{pbini_gap}-\eqref{cond_dyn_gap} in the intracellular and extracellular media and  omitting the superscript $ \  \widehat{\cdot} \ $ of the dimensionless variables, we obtain the following non-dimensional form: 
\begin{subequations}
\begin{align}
-\nabla\cdot\left( \mathrm{M}_{i}^{\e}\nabla u_{i,\e}^{k}\right)  &=0 & \text{ in } \Omega_{i,\e,T}^{k}:=(0,T)\times\Omega_{i,\e}^{k}, 
\label{trid_intra}
\\ -\nabla\cdot\left( \mathrm{M}_{e}^{\e}\nabla u_{e,\e}\right)  &=0 & \text{ in } \Omega_{e,\e,T}:=(0,T)\times\Omega_{e,\e}, 
\label{trid_extra}
\\u_{i,\e}^{k}-u_{e,\e}&=v_{\e}^{k} &\ \text{on} \ \Gamma_{\e,T}^{k}:=(0,T)\times\Gamma_{\e}^{k},
\label{trid_v}
\\-\mathrm{M}_{i}^{\e}\nabla u_{i,\e}^{k} \cdot n_i^{k}=\mathrm{M}_{e}^{\e}\nabla u_{e,\e} \cdot n_{e} & =\I_{m}^{k} &\text{ on } \Gamma_{\e,T}^{k},
\label{trid_cont_sarco}
\\ \e\left( \pt_{t} v_{\e}^{k}+\I_{ion}\left(v_{\e}^{k},w_{\e}^{k}\right) -\I_{app,\e}^{k}\right) &=\I_{m}^{k} &\ \text{on} \ \Gamma_{\e,T}^{k},
\label{trid_bord_sarco}
\\ \pt_{t} w_{\e}^{k}-H\left(v_{\e}^{k},w_{\e}^{k}\right) &=0 & \text{ on }  \Gamma_{\e,T}^{k},
\label{trid_dyn_sarco}
\\u_{i,\e}^{1}-u_{i,\e}^{2}&=s_{\e} &\ \text{on} \ \Gamma_{\e,T}^{1,2}:=(0,T)\times\Gamma_{\e}^{1,2},
\label{trid_s}
\\-\mathrm{M}_{i}^{\e}\nabla u_{i,\e}^{1} \cdot n_i^{1}=\mathrm{M}_{i}^{\e}\nabla u_{i,\e}^{2} \cdot n_i^{2}& =\I_{1,2} &\ \text{on} \ \Gamma_{\e,T}^{1,2},
\label{trid_cont_gap}
\\ \frac{\e}{2}\left( \pt_{t} s_{\e}+\I_{gap}\left(s_{\e}\right) \right) &=\I_{1,2} &\ \text{on} \ \Gamma_{\e,T}^{1,2},
\label{trid_bord_gap}
\end{align}
\label{pbscale_gap}
\end{subequations}
with $k=1,2$ and each equation corresponds to the following sense:
\eqref{trid_intra} Intra quasi-stationary conduction, \eqref{trid_extra} Extra quasi-stationary conduction, \eqref{trid_v} Transmembrane potential, \eqref{trid_cont_sarco} Continuity equation at cell membrane, \eqref{trid_bord_sarco} Reaction condition at the corresponding cell membrane, \eqref{trid_dyn_sarco} Dynamic coupling, \eqref{trid_s} Gap junction potential, \eqref{trid_cont_gap} Continuity equation at gap junction, \eqref{trid_bord_sarco} Reaction condition at gap junction. 

 Observe that the tridomain equations \eqref{trid_intra}-\eqref{trid_extra} are invariant with respect to the above scaling. We define now the rescaled electrical potential as follows:
$$u_{i,\e}^{k}(t,x):=u_{i}^{k}\left( t, x,\frac{x}{\e}\right), \ \ u_{e,\e}(t,x):= u_{e}\left( t, x,\frac{x}{\e}\right), \text{ for } k=1,2.$$
Analogously, we obtain the rescaled transmembrane potential $v_{\e}^{k},$  the rescaled gap junction potential $s_{\e}$ and the corresponding gating variable $w_{\e}^{k}$ for $k=1,2.$ Furthermore, the conductivity tensors are considered dependent both on the slow and fast variables, i.e. for $j=i,e,$ we have 
\begin{equation}
\mathrm{M}_{j}^{\e}(x):=\mathrm{M}_{j}\left( x,\frac{x}{\e}\right),
\label{M_ie_gap}
\end{equation}
satisfying the elliptic and periodicity conditions: there exist constants $\alpha, \beta \in \R,$ such that $0<\alpha<\beta$ and for all $\lambda\in \R^d:$
\begin{subequations}
\begin{align}
&\mathrm{M}_j\lambda\cdot\lambda \geq \alpha\abs{\lambda}^2, 
\\& \abs{\mathrm{M}_j\lambda}\leq \beta \abs{\lambda},
\\& \mathrm{M}_j \ \mathbf{y}\text{-periodic},\text{ for }  j=i,e.
\end{align}
\label{A_Mie_gap}
\end{subequations}
\begin{rem} Finally, we assume that each $\mathrm{M}_j$ is symmetric: $\mathrm{M}_j^{T}=\mathrm{M}_j.$
\end{rem}
We complete system \eqref{pbscale_gap} with no-flux boundary conditions on $\pt_{\text{ext}} \Omega$: 
\begin{equation*}
\left( \mathrm{M}_{i}^{\e}\nabla u_{i,\e}^{k}\right) \cdot \mathbf{n}=\left( \mathrm{M}_{e}^{\e}\nabla u_{e,\e}\right) \cdot\mathbf{n}=0 \ \text{ on } \  (0,T)\times \pt_{\text{ext}} \Omega,
\end{equation*}
where $k=1,2$ and $\mathbf{n}$ is the outward unit normal to the exterior boundary of $\Omega.$ We impose initial conditions on transmembrane potential $v_{\e}^{k},$ gap junction potential $s_{\e}$ and gating variable $w_{\e}^{k}$ as follows: 
\begin{equation}
\begin{aligned}
& v_{\e}^{k}(0,x)=v_{0,\e}^{k}(x), \ w_{\e}^{k}(0,x)=w_{0,\e}^{k}(x) & \text{ a.e. on } \Gamma_{\e,T}^{k}, 
\\ & \text{and } s_{\e}(0,x)=s_{0,\e}(x) & \text{ a.e. on } \Gamma_{\e,T}^{1,2},
\end{aligned}
\label{cond_ini_vws_gap}
\end{equation}
with $k=1,2.$ \\
We mention some assumptions on the ionic functions, the source term and the initial data.\\
\textbf{Assumptions on the ionic functions.} The ionic current $\I_{ion}(v^{k},w^{k})$ at each cell membrane $\Gamma^{k}$ can be decomposed into $\mathrm{I}_{a,ion}\left( v^{k} \right)$ and $\mathrm{I}_{b,ion}^{k}\left( w^{k}\right) ,$ where $\I_{ion}\left(v^{k},w^{k}\right)=\mathrm{I}_{a,ion}\left(v^{k}\right) +\mathrm{I}_{b,ion}\left(w^{k}\right) $ with $k=1,2.$ Furthermore, the nonlinear function $\mathrm{I}_{a,ion}: \R \rightarrow \R$ is considered as a $C^1$ function and the functions $\mathrm{I}_{b,ion}: \R \rightarrow \R$  and $H : \R^2 \rightarrow \R$ are considered as linear functions. Also, we assume that there exists $r\in (2,+\infty)$ and constants $\alpha_1,\alpha_2,\alpha_3, \alpha_4, \alpha_5, C>0$ and $\beta_1>0, \beta_2\geq0$ such that:

\begin{subequations}
\begin{align}
&\dfrac{1}{\alpha_1} \abs{v}^{r-1}\leq \abs{\mathrm{I}_{a,ion}\left(v\right)}\leq \alpha_1\left(\abs{v}^{r-1}+1\right), \,\abs{\mathrm{I}_{b,ion}\left(w\right)}\leq \alpha_2(\abs{w}+1), 
\label{A_I_ab} 
\\& \abs{H(v,w)}\leq \alpha_3(\abs{v}+\abs{w}+1),\text{ and }
    \mathrm{I}_{b,ion}\left(w\right)v-\alpha_4 H(v,w)w\geq \alpha_5 \abs{w}^2,
\label{A_H_Ib_a}
\\& \tilde{\mathrm{I}}_{a,ion} : v\mapsto \mathrm{I}_{a,ion}(v)+\beta_1 v+\beta_2 \text{ is strictly increasing with } \lim \limits_{v\rightarrow 0} \tilde{\mathrm{I}}_{a,ion}(v)/v=0,
\label{A_tildeI_a_1} 
\\& \forall v,v' \in \R,\,\,\left(\tilde{\mathrm{I}}_{a,ion}(v)-\tilde{\mathrm{I}}_{a,ion}(v') \right)(v-v')\geq \dfrac{1}{C} \left(1+\abs{v}+\abs{v'} \right)^{r-2} \abs{v-v'}^{2},
\label{A_tildeI_a_2} 
\end{align}
\label{A_H_I_gap}
\end{subequations}
with $(v,w):=\left(v^{k},w^{k}\right)$ for $k=1,2.$
\begin{rem}
In the mathematical analysis of bidomain equations, several paths have been followed in the literature according to the definition of the ionic currents. We summarize below the phenomenological models:
\begin{enumerate}
\item[]Other non-physiological models have been introduced as approximations of ion current models. They can be used in large problems because they are typically small and fast to solve, although they are less flexible in their response to variations in cellular properties such as concentrations or cell size. We take in this paper the FitzHugh-Nagumo model \cite{fitz,nagumo} that satisfies assumptions \eqref{A_H_I_gap} which reads as
\begin{subequations}
  \begin{align}
  H\left(v,w\right)&= a_1 v-b_1w, \\  \I_{ion}\left( v,w\right)&=\left[\rho v (1-v)\left( v-\theta\right) \right] -\rho w:= \mathrm{I}_{a,ion}\left(v\right)+\mathrm{I}_{b,ion}\left(w\right)
  \end{align}
  \label{ionic_model_gap}
  \end{subequations}
  where $a_1, b_1, \rho, \theta$ are given parameters with $a_1,b_1> 0, \ \rho<0$ and $\theta \in (0,1).$  According to this model, the functions $\I_{ion}$ and $H$ are continuous and the non-linearity $\mathrm{I}_{a,ion}$ is of cubic growth at infinity then the most appropriate value is $r=4.$ Using Young's inequality, we have
  \begin{equation}
  \abs{v}^2\leq \dfrac{2\abs{v}^3}{3}+\dfrac{1}{3}, \quad \abs{v}\leq \dfrac{\abs{v}^3}{3}+\dfrac{2}{3}, \quad \abs{v}\leq \dfrac{\abs{v}^2}{2}+\dfrac{1}{2}
  \end{equation}
and then assumption \eqref{A_I_ab} holds for $r=4:$
\begin{align*}
&\abs{\mathrm{I}_{a,ion}\left(v\right)}=\abs{\rho v (1-v)\left( v-\theta\right)}\leq \left( \dfrac{2}{3}\theta+\dfrac{1}{3}(1+\theta)\right)\abs{\rho} +\left(\dfrac{1}{3}\theta+\dfrac{2}{3}(1+\theta)+1 \right)\abs{\rho}\abs{v}^3,
\\ & \abs{\mathrm{I}_{b,ion}\left(w\right)}=\abs{\rho} \abs{w},
\\ & \abs{H\left(v,w\right)}=\abs{a_1 v-b_1 w} \leq a_1 \abs{v}+b_1 \abs{w}.
\end{align*}
Now, we compute the function $E(u,v):=\mathrm{I}_{b,ion}\left(w\right)v-\alpha_4 H(v,w)w$ defined in $\R^2.$ So, the second assumption \eqref{A_H_Ib_a} holds with $\alpha_4=-\dfrac{\rho}{a_1}:$
\begin{equation}
E(u,v)=\dfrac{\rho}{a_1} w^2.
\end{equation}
Moreover, the conditions \eqref{A_tildeI_a_1}-\eqref{A_tildeI_a_2} are automatically satisfied by any cubic polynomial $\I_{ion}$ with positive leading coefficient. We end this remark by mentioning other reduced ionic models: the Roger-McCulloch model \cite{roger} and the Aliev-Panfilov model \cite{aliev}, may consider more general that the previous model but still rise some mathematical difficulties. Furthermore, the Mitchell-Schaeffer model \cite{mitchell} has been studied in \cite{boulakia,kunisch} and its regularized version have a very specific structure.  In particular, no proof of uniqueness of solutions for these models exists in the literature. 
\end{enumerate}
\end{rem}
Now, we represent the gap junction $\Gamma_{\e}^{1,2}$ between intra-neighboring cells by a passive membrane:
\begin{equation}
\I_{gap}(s)=G_{gap}s,
\label{gap_model}
\end{equation}
where $G_{gap}=\frac{1}{R_{gap}}$ is the conductance of the gap junctions. A discussion of the modeling of the gap junctions is given in \cite{hogues}.\\
\textbf{Assumptions on the source term.} There exists a constant $C$ independent of $\e$ such that the source term $\I_{app,\e}^{k}$ satisfies the following estimation for $k=1,2$:
\begin{equation}
\norm{\e^{1/2}\I_{app,\e}^{k}}_{L^{2}(\Gamma_{\e,T}^{k})}\leq C.
\label{A_iapp_gap}
\end{equation}  
\textbf{Assumptions on the initial data.} The initial condition $v_{0,\e}^{k},$ $s_{0,\e}$ and $w_{0,\e}^{k}$ satisfy the following estimation:
\begin{equation}
\sum\limits_{k=1,2}\norm{\e^{1/r}v_{0,\e}^{k}}_{L^{r}(\Gamma_{\e}^{k})}+\norm{\e^{1/2}s_{0,\e}}_{L^{2}(\Gamma_{\e}^{1,2})}+\sum\limits_{k=1,2}\norm{\e^{1/2}w_{0,\e}^{k}}_{L^{2}(\Gamma_{\e}^{k})}\leq C,
\label{A_vw0_gap}
\end{equation}
for some constant $C$ independent of $\e.$ Moreover, $v_{0,\e}^{k},$ $s_{0,\e}$ and $w_{0,\e}^{k}$ are assumed to be traces of uniformly bounded sequences in $C^{1}(\overline{\Omega})$ with $k=1,2.$

  Finally,  one can observe that Equations in \eqref{pbscale_gap} are invariant under the change of $u_{i,\e}^{k},$ $k=1,2$ and $u_{e,\e}$ into $u_{i,\e}^{k}+c,$ $u_{e,\e}+c,$ for any $c\in\R.$ Therefore, we may impose the following normalization condition:
 \begin{equation}
 \int_{\Omega_{e,\e}} u_{e,\e} \ dx=0, \text{ for a.e. } t\in(0,T).
 \label{normalization}
 \end{equation}
 
 \section{Main results}\label{main_results_trid}
 In this part, we highlight our main results obtained in our paper. First, we define the weak solutions of the microscopic tridomain model. Next, we find a priori estimates and we supply our existence and uniqueness results by using Faedo-Galerkin method, compactness argument and monotonicity.
 
  We start by stating the weak formulation of the microscopic tridomain model as given in the following definition.
\begin{defi}[Weak formulation of microscopic system] A weak solution to problem \eqref{pbscale_gap}-\eqref{cond_ini_vws_gap} is a collection $(u_{i,\e}^{1},u_{i,\e}^{2},u_{e,\e},w_{\e}^{1},w_{\e}^{2})$ of functions satisfying the following conditions:


\begin{enumerate}[label=(\Alph*)]
\item (Algebraic relation).
\begin{equation*}
\begin{aligned}
v_{\e}^{k}&=(u_{i,\e}^{k}-u_{e,\e})\vert_{\Gamma_{\e,T}^{k}} &\ \text{a.e. on} \ \Gamma_{\e,T}^{k}, \text{ for } k=1,2,
\\ s_{\e}&=(u_{i,\e}^{1}-u_{i,\e}^{2})\vert_{\Gamma_{\e,T}^{1,2}} &\ \text{a.e. on} \ \Gamma_{\e,T}^{1,2}.
\end{aligned}
\end{equation*}
\item (Regularity).
\begin{equation*}
\begin{aligned}
 & u_{i,\e}^{k}\in L^{2}\left(0,T;H^{1}\left( \Omega_{i,\e}^{k}\right)\right), \quad u_{e}^{\e}\in L^{2}\left(0,T;H^{1}(\Omega_{e,\e})\right),
 \\& \int_{\Omega_{e,\e}} u_{e,\e}(t,x) \ dx=0, \text{ for a.e. } t\in(0,T),
\\ & v_{\e}^{k}\in L^{2}\left( 0,T;H^{1/2}\left(\Gamma_{\e}^{k}\right)\right)\cap L^{r}\left(\Gamma_{\e,T}^{k}\right), \ r\in (2,+\infty)
\\ & s_{\e}\in L^{2}\left(\Gamma_{\e,T}^{1,2}\right), \quad w_{\e}^{k} \in L^{2}(\Gamma_{\e,T}^{k}),
\\& \pt_t v_{\e}^{k} \in L^{2}\left( 0,T;H^{-1/2}\left(\Gamma_{\e}^{k}\right)\right)+L^{r/(r-1)}\left(\Gamma_{\e,T}^{k}\right),
\\ & \pt_t s_{\e} \in L^{2}(\Gamma_{\e,T}^{1,2}), \quad \pt_t w_{\e}^{k} \in L^{2}(\Gamma_{\e,T}^{k}) \text{ for } k=1,2.
\end{aligned}
\end{equation*}
\item (Initial conditions).
\begin{equation*}
\begin{aligned}
& v_{\e}^{k}(0,x)=v_{0,\e}^{k}(x), \ w_{\e}^{k}(0,x)=w_{0,\e}^{k}(x) & \text{ a.e. on } \Gamma_{\e}^{k}, 
\\ & \text{and } s_{\e}(0,x)=s_{0,\e}(x) & \text{ a.e. on } \Gamma_{\e}^{1,2}.
\end{aligned}
\end{equation*}
\item (Variational equations).
\begin{equation}
\begin{aligned}
&\sum \limits_{k=1,2}\iint_{\Gamma_{\e,T}^{k}} \e\pt_t v_{\e}^{k} \psi^{k} \ d\sigma_xdt+\frac{1}{2}\iint_{\Gamma_{\e,T}^{1,2}} \e\pt_t s_{\e} \Psi \ d\sigma_xdt 
\\& \quad +\sum \limits_{k=1,2}\int_{\Omega_{i,\e,T}^{k}}\mathrm{M}_{i}^{\e}\nabla u_{i,\e}^{k}\cdot\nabla\varphi_{i}^{k} \ dxdt+\int_{\Omega_{e,\e,T}}\mathrm{M}_{e}^{\e}\nabla u_{e,\e}\cdot\nabla\varphi_{e} \ dxdt
\\& \quad +\sum \limits_{k=1,2}\iint_{\Gamma_{\e,T}^{k}} \e\I_{ion}\left(v_{\e}^{k},w_{\e}^{k}\right)\psi^{k} \ d\sigma_xdt+\frac{1}{2}\iint_{\Gamma_{\e,T}^{1,2}} \e \I_{gap}(s_{\e})\Psi \ d\sigma_xdt
\\& =\sum \limits_{k=1,2}\iint_{\Gamma_{\e,T}^{k}} \e\I_{app,\e}^{k}\psi^{k} \ d\sigma_xdt
\end{aligned}
\label{Fv_ike_ini_gap}
\end{equation}
\begin{equation}
\iint_{\Gamma_{\e,T}^{k}} \pt_t w_{\e}^{k}e^{k} \ d\sigma_xdt=\iint_{\Gamma_{\e,T}^{k}} H\left(v_{\e}^{k},w_{\e}^{k}\right) e^{k} \ d\sigma_xdt
\label{Fv_d_ini_gap}
\end{equation}
\end{enumerate}
for all $\varphi_{i}^{k}\in L^{2}\left(0,T;H^{1}\left( \Omega_{i,\e}^{k}\right)\right),$ $\varphi_{e}\in L^{2}\left(0,T;H^{1}(\Omega_{e,\e})\right)$ with 
\begin{itemize}
\item $\psi^{k}=\psi_{i}^{k}-\psi_{e}^{k}:=\left(\varphi_{i}^{k}-\varphi_{e}\right)\vert_{\Gamma_{\e,T}^{k}} \in L^{2}\left( 0,T;H^{1/2}\left(\Gamma_{\e}^{k}\right)\right)\cap L^{r}\left(\Gamma_{\e,T}^{k}\right)$ for $k=1,2,$
\item $\Psi=\Psi_{i}^{1}-\Psi_{i}^{2}:=\left(\varphi_{i}^{1}-\varphi_{i}^{2}\right)\vert_{\Gamma_{\e,T}^{1,2}}\in  L^{2}(\Gamma_{\e,T}^{1,2}),$
\item $e^{k}\in L^{2}(\Gamma_{\e,T}^{k})$ for $k=1,2.$
\end{itemize}
\label{Fv_gap} 

\end{defi}

\begin{rem} Due to Lions-Magenes theorem $($see \cite{boyer} p. 101$),$ the following injection 
\begin{align*}
&\mathcal{V}:=\left\lbrace  u\in L^{2}\left( 0,T;H^{1/2}\left(\Gamma_{\e}^{k}\right)\right)\cap L^{r}\left(\Gamma_{\e,T}^{k}\right) \text{ and } \pt_t u \in L^{2}\left( 0,T;H^{-1/2}\left(\Gamma_{\e}^{k}\right)\right)+L^{r/(r-1)}\left(\Gamma_{\e,T}^{k}\right)\right\rbrace 
\\ & \qquad  \subset C^0\left([0,T];L^2(\Gamma_{\e})\right), \text{ for } k=1,2  
\end{align*}
is continuous with $r\in(2,+\infty)$. Then, $v_{\e}^{k}$ $\in C^0\left([0,T];L^2(\Gamma_{\e}^{k})\right)$ for $k=1,2$. Therefore, the initial data of $v_{\e}^{k}$ for $k=1,2$ in Definition \ref{Fv_gap} is well defined. In the same manner, the initial condition  on $s_{\e}$ and on $w_{\e}^{k}$ for $k=1,2$  makes sense.
\end{rem}

\begin{thm}[Microscopic Tridomain Model] Assume that the conditions \eqref{A_Mie_gap}-\eqref{A_vw0_gap} hold. Then, System \eqref{pbscale_gap}-\eqref{cond_ini_vws_gap} possesses a unique weak solution in the sense of Definition \ref{Fv_gap} for every fixed $\e>0$.  
 
 Furthermore, this solution verifies the following energy estimates: there exists constants $C_1, C_2, C_3,C_4,$ independent of $\e$ such that:
\begin{equation}
\sum\limits_{k=1,2} \norm{\sqrt{\e}v_{\e}^{k}}_{L^{\infty}\left(0,T;L^2(\Gamma_{\e}^{k})\right)}^2+\sum\limits_{k=1,2} \norm{\sqrt{\e}w_{\e}^{k}}_{L^{\infty}\left(0,T;L^2(\Gamma_{\e}^{k})\right)}^2+\norm{\sqrt{\e} s_{\e}}_{L^{\infty}\left(0,T;L^2(\Gamma_{\e}^{1,2})\right)}^2\leq C_1
\label{E_vw_gap}
\end{equation}

\begin{equation}
\sum\limits_{k=1,2}\norm{u_{i,\e}^{k}}_{L^{2}\left(0,T;H^{1}\left( \Omega_{i,\e}^{k}\right) \right)}+\norm{u_{e}^{\e}}_{L^{2}\left(0,T;H^{1}\left(\Omega_{e,\e}\right) \right)}\leq C_2,
\label{E_u_gap}
\end{equation}

\begin{equation}
\sum\limits_{k=1,2}\norm{\e^{1/r}v_{\e}^{k}}_{L^{r}(\Gamma_{\e,T}^{k})}\leq C_3 \text{ and } \sum\limits_{k=1,2}\norm{\e^{(r-1)/r}\mathrm{I}_{a,ion}(v_{\e}^{k})}_{L^{r/(r-1)}(\Gamma_{\e,T}^{k})}\leq C_4.
\label{E_vr_gap}
\end{equation}
Moreover, if $v_{\e,0}^{k} \in H^{1/2}(\Gamma_{\e}^{k})\cap L^{r}(\Gamma_\e^{k}),$ $k=1,2,$ then there exists a constant $C_5$  independent of $\e$ such that:
\begin{equation}
\sum\limits_{k=1,2} \norm{\sqrt{\e}\pt_t v_{e}^{k}}_{L^2(\Gamma_{\e,T}^{k})}^2+\sum\limits_{k=1,2} \norm{\sqrt{\e}\pt_t w_{\e}^{k}}_{L^2(\Gamma_{\e,T}^{k})}^2+\norm{\sqrt{\e} \pt_t s_{\e}}_{L^2(\Gamma_{\e,T}^{1,2})}^2\leq C_5.
\label{E_dtv_gap}
\end{equation}

\label{thm_micro_gap}
\end{thm}

The proof of Theorem \ref{thm_micro_gap} is treated in Section \ref{Exist_sol_gap}.

\begin{rem} The authors in \cite{bendunf19,BaderDev,BaderUnf} treated the microscopic bidomain problem where the gap junction is ignored. They considered that there are only  intra- and extracellular media separated by the membrane (sarcolemma). Comparing to \cite{bendunf19}, the microscopic tridomain model in our work consists of three elliptic equations coupled through three boundary conditions, two on each cell membrane and one on the gap junction which separates between two intracellular media. 
\end{rem}

\section{Existence and Uniqueness of solutions for the microscopic tridomain model}\label{Exist_sol_gap}
 This section is devoted to proving existence and uniqueness of solutions to the heterogeneuous microscopic tridomain model presented in Section \ref{geotrid} for fixed $\e>0.$ The proof of Theorem \ref{thm_micro_gap} is based on the Faedo-Galerkin method and carried out in several steps:
\begin{itemize}
\item[$\bullet$] Construction of the basis on the intra- and extracellular domains.
\item[$\bullet$] Construction and local existence of approximate solutions.
\item[$\bullet$] Find some a priori estimates of the approximate solutions.
\item[$\bullet$] Existence and uniqueness of solution to the microscopic tridomain model.
\end{itemize}
 We refer the reader to the well-posedness results for weak solutions of the microscopic bidomain model, established in \cite{char09,bendunf19} by using a Faedo-Galerkin technique. See also \cite{bend,boulakia} for a similar approach, based on a parabolic regularization technique.

 In this proof, we will remove the $\e$-dependence in the solution $\left(u_{i,\e}^{1},u_{i,\e}^{2},u_{e,\e},v_{\e}^{1},v_{\e}^{2},s_{\e},w_{\e}^{1},w_{\e}^{2} \right)$ for simplification of notation. The demonstration is described as follows:
\subsection*{Step 1: Construction of the basis}
We first consider functions $\phi,$ $\tilde{\phi}\in C^{0}(\overline{\Omega}_{i,\e}^{k})$ and we let $\V_{0,i}^{k}$ denote the completion of $C^{0}(\overline{\Omega}_{i,\e}^{k})$ under the norm induced by the inner product $\langle\cdot,\cdot\rangle_{\V_{0,i}^k}$ which defined by
\begin{equation*}
\langle\Theta, \tilde{\Theta}\rangle_{\V_{0,i}^{k}}:=\displaystyle \int_{\Omega_{i,\e}^{k}}\phi \tilde{\phi} \ dx+\int_{\Gamma_{\e}^{k}} \phi\vert_{\Gamma_{\e}^{k}}\tilde{\phi}\vert_{\Gamma_{\e}^{k}}d\sigma+\int_{\Gamma_{\e}^{1,2}} \phi\vert_{\Gamma_{\e}^{1,2}}\tilde{\phi}\vert_{\Gamma_{\e}^{1,2}}d\sigma, \text{ for } k=1,2,
\end{equation*}
where $\Theta={}^{t}\left( \phi \quad \phi\vert_{\Gamma_{\e}^{k}} \quad \phi\vert_{\Gamma_{\e}^{1,2}} \right),$ $\tilde{\Theta}={}^{t}\left( \tilde{\phi} \quad \tilde{\phi}\vert_{\Gamma_{\e}^{k}} \quad \tilde{\phi}\vert_{\Gamma_{\e}^{1,2}} \right).$ Similarly,  for functions $\phi,$ $\tilde{\phi}\in C^{1}(\overline{\Omega}_{i,\e}^{k})$ and we let $\V_{1,i}^{k}$ denote the completion of $C^{1}(\overline{\Omega}_{i,\e}^{k})$ under the norm induced by the inner product $\langle\cdot,\cdot\rangle_{\V_{1,i}^{k}}$ which defined by
\begin{align*}
\langle\Theta, \tilde{\Theta}\rangle_{\V_{1,i}^{k}}
&:=\displaystyle \int_{\Omega_{i,\e}^{k}}\mathrm{M}_{i}^{\e}\nabla \phi\cdot\nabla\tilde{\phi} \ dx+\int_{\Gamma_{\e}^{k}} \phi\vert_{\Gamma_{\e}^{k}}\tilde{\phi}\vert_{\Gamma_{\e}^{k}} \ d\sigma+\int_{\Gamma_{\e}^{k}} \nabla_{\Gamma_{\e}^{k}}\phi\cdot \nabla_{\Gamma_{\e}^{k}}\tilde{\phi} \ d\sigma
\\& +\int_{\Gamma_{\e}^{1,2}} \phi\vert_{\Gamma_{\e}^{1,2}}\tilde{\phi}\vert_{\Gamma_{\e}^{1,2}} \ d\sigma+\int_{\Gamma_{\e}^{1,2}} \nabla_{\Gamma_{\e}^{1,2}}\phi\cdot \nabla_{\Gamma_{\e}^{1,2}}\tilde{\phi} \ d\sigma, \text{ for } k=1,2
\end{align*}
where $\nabla_{\Gamma}$ denotes the tangential gradient operator on $\Gamma$ $(\Gamma:=\Gamma_{\e}^{k},\Gamma_{\e}^{1,2}).$  We note that the following injections hold:
\begin{equation*}
\V_{0,i}^{k}\subset L^{2}(\Omega_{i,\e}^{k}), \text{ and } \V_{1,i}^{k}\subset H^{1}(\Omega_{i,\e}^{k}).
\end{equation*}
Moreover, the injection from $\V_{1,i}^{k}$ to $\V_{0,i}^{k}$ is continous and compact for $k=1,2$. We refer the reader to \cite{ciprian,racke} for similar approaches. It follows from a well-known result (see e.g. \cite{temam} p. 54) that the  bilinear form $$a(\Theta, \tilde{\Theta}):=\langle\Theta, \tilde{\Theta}\rangle_{\V_{1,i}^{k}}$$ 
defines a strictly positive self adjoint unbounded operator $\mathcal{B}_{i}^{k}$
\begin{equation*}
\mathcal{B}_{i}^{k} : D(\mathcal{B}_{i}^{k}) = \lbrace\Theta\in \V_{1,i}^{k} : \mathcal{B}_{i}^{k}\Theta \in \V_{1,i}^{k}\rbrace  \rightarrow \V_{0,i}^{k}
\end{equation*}
such that, for any $\tilde{\Theta}\in \V_{1,i}^{k},$ we have $\langle\mathcal{B}_{i}^{k}\Theta, \tilde{\Theta}\rangle_{\V_{0,i}^{k}} = a(\Theta, \tilde{\Theta}).$ Thus, for $n\in \mathbb{N},$ we take a complete system of eigenfunctions $\lbrace\Theta_{i,n}^{k}={}^{t}\left( \phi_{i,n}^{k} \quad \psi_{i,n}^{k} \quad \Psi_{i,n}^{k} \right)\rbrace_{n}$ of the problem $$\mathcal{B}_{i}^{k}\Theta_{i,n}^{k}=\lambda_{n}\Theta_{i,n}^{k}, \text{ in } \V_{0,i}^{k}, \text{ for } k=1,2,$$ 
with \begin{itemize}
\item $\lbrace\lambda_{n}\rbrace_n$ be a sequence such that $0<\lambda_{1}\leq \lambda_{2},\dots,\lambda_{n}\rightarrow \infty$ as $n\rightarrow \infty.$
\item $\Theta_{i,n}^{k}\in D(\mathcal{B}_{i}^{k}),$ $\psi_{i,n}^{k}:=\phi_{i,n}^{k}\vert_{\Gamma_\e^{k}}$ and $\Psi_{i,n}^{k}:=\phi_{i,n}^{k}\vert_{\Gamma_\e^{1,2}}$ where $\phi_{i,n}^{k}, \ \psi_{i,n}^{k}$ and $\Psi_{i,n}^{k}$ are regular enough for $k=1,2$.
\end{itemize}   
Moreover, the eigenvectors $\lbrace\Theta_{i,n}^{k}\rbrace_{n}$ turn out to form an orthogonal basis in $\V_{1,i}^{k}$ and $\V_{0,i}^{k},$ and they may be assumed to be normalized in the norm of $\V_{0,i}^{k}$ for $k=1,2$. Since $C^{1}(\overline{\Omega}_{i,\e}^{k})\subset \V_{1,i}^{k} \subset H^{1}(\Omega_{i,\e}^{k})$ and $C^{1}(\overline{\Omega}_{i,\e}^{k})$ is dense in $H^{1}(\Omega_{i,\e}^{k})$ then $\V_{1,i}^{k}$ is dense in $H^{1}(\Omega_{i,\e}^{k})$ for the $H^{1}$-norm. Therefore, $\lbrace\Theta_{i,n}^{k}\rbrace_{n}$ is a basis in $H^{1}(\Omega_{i,\e}^{k})$ for the $H^{1}$-norm.
 
 On the other hand, we consider a basis $\lbrace \zeta_{n}^{k}\rbrace_{n}, n\in\mathbb{N}$ that is orthonormal in $L^{2}(\Gamma_{\e}^{k})$ and orthogonal in $H^{1}(\Gamma_{\e}^{k})$ and we set the spaces 
\begin{equation*}
\P_{i,\ell}^{k}=\text{span}\lbrace\Theta_{i,1}^{k},\dots, \Theta_{i,\ell}^{k}\rbrace, \quad \P_{i,\infty}^{k}=\overset{\infty}{\underset{\ell=1}{\bigcup}}\P_{i,\ell}^{k},
\end{equation*}
\begin{equation*}
\K_{i,\ell}^{k}=\text{span}\lbrace\zeta_{1}^{k},\dots, \zeta_{\ell}^{k}\rbrace, \quad \K_{i,\infty}^{k}=\overset{\infty}{\underset{\ell=1}{\bigcup}}\K_{i,\ell}^{k},
\end{equation*}
where $\P_{i,\infty}^{k}$ and $\K_{i,\infty}^{k}$ are respectively dense subspaces of $\V_{1,i}^{k}$ and $H^{1}(\Gamma_{\e}^{k})$ for $k=1,2.$
\begin{rem}
Analogously, we construct a basis on the extracellular domain. We let $\V_{p,e}$ denote the completion of $C^{p}(\overline{\Omega}_{e,\e})$ under the norm induced by the inner product $\langle\cdot,\cdot\rangle_{\V_{p,e}}$  for $\phi,$ $\tilde{\phi}\in C^{p}(\overline{\Omega}_{e,\e}),$ $p=0,1$  which respectively defined by
\begin{equation*}
\langle\Theta', \tilde{\Theta}'\rangle_{\V_{0,e}}:=\displaystyle \int_{\Omega_{e,\e}}\phi \tilde{\phi} \ dx+\sum \limits_{k=1,2}\int_{\Gamma_{\e}^{k}} \phi\vert_{\Gamma_{\e}^{k}}\tilde{\phi}\vert_{\Gamma_{\e}^{k}}d\sigma,
\end{equation*}
and
\begin{align*}
\langle\Theta', \tilde{\Theta}'\rangle_{\V_{1,e}}
&:=\displaystyle \int_{\Omega_{e,\e}}\mathrm{M}_{e}^{\e}\nabla \phi\cdot\nabla\tilde{\phi} \ dx+\sum \limits_{k=1,2}\left[ \int_{\Gamma_{\e}^{k}} \phi\vert_{\Gamma_{\e}^{k}}\tilde{\phi}\vert_{\Gamma_{\e}^{k}} \ d\sigma+\int_{\Gamma_{\e}^{k}} \nabla_{\Gamma_{\e}^{k}}\phi\cdot \nabla_{\Gamma_{\e}^{k}}\tilde{\phi} \ d\sigma\right], 
\end{align*}
where $\Theta'={}^{t}\left( \phi \quad \phi\vert_{\Gamma_{\e}^{1}} \quad \phi\vert_{\Gamma_{\e}^{2}}  \right),$ $\tilde{\Theta}'={}^{t}\left( \tilde{\phi} \quad \tilde{\phi}\vert_{\Gamma_{\e}^{1}} \quad \tilde{\phi}\vert_{\Gamma_{\e}^{2}} \right).$ Similarly, we take a complete basis which is orthogonal in $\V_{1,e}$ and orthonormal in $\V_{0,e}$ and we set the spaces 
\begin{equation*}
\P_{e,\ell}=\text{span}\lbrace\Theta_{e,1},\dots, \Theta_{e,\ell}\rbrace, \quad \P_{e,\infty}=\overset{\infty}{\underset{\ell=1}{\bigcup}}\P_{e,\ell},
\end{equation*}
where $\P_{e,\infty}$ is a dense subspace of $\V_{1,e}.$
\end{rem}
\subsection*{Step 2: Construction and local existence of approximate solutions}

 Supplied with the basis introduced in the first step, 
 we look for the approximate solutions  as sequences $\lbrace u_{i,n}^{k}\rbrace_{n>1},$ $\lbrace u_{e,n}\rbrace_{n>1}$ and $\lbrace w_{n}^{k}\rbrace_{n>1},$ $k=1,2$ defined for $t>0$ and $x\in \Omega$ by:
\begin{equation}
\begin{aligned}
& U_{i}^{k}=\left(
\begin{array}{c}
u_{i,n}^{k} \\ \\ \ov{u}_{i,n}^{k} \\ \\ \ovo{u}_{i,n}^{k}
\end{array}
\right)
:=\sum \limits_{\ell=1}^{n} d_{i,\ell}^{k}(t)\left(
\begin{array}{c}
\phi_{i,\ell}^{k} \\ \\ \psi_{i,\ell}^{k} \\ \\ \Psi_{i,\ell}^{k}
\end{array}
\right), \quad U_{e}=\left(
\begin{array}{c}
u_{e,n} \\ \\ \ov{u}_{e,n}^{1} \\ \\ \ov{u}_{e,n}^{2}
\end{array}
\right)
:=\sum \limits_{\ell=1}^{n} d_{e,\ell}(t)\left(
\begin{array}{c}
\phi_{e,\ell} \\ \\ \psi_{e,\ell}^{1} \\ \\ \psi_{e,\ell}^{2}
\end{array}
\right) 
\\ & \quad \text{and }  w_{n}^{k}:=\sum \limits_{\ell=1}^{n} c_{\ell}^{k}(t)\zeta_{\ell}^{k}(x),
\end{aligned}
\label{Approx_sol}
\end{equation}
with $\phi_{i,\ell}^{k}\vert_{\Gamma_{\e}^{k}}=\psi_{i,\ell}^{k},$ $\phi_{i,\ell}^{k}\vert_{\Gamma_{\e}^{1,2}}=\Psi_{i,\ell}^{k}$ and $\phi_{e,\ell}\vert_{\Gamma_{\e}^{k}}=\psi_{e,\ell}^k$ for $k=1,2.$ To apply the Faedo-Galerkin scheme, we first regularize the microscopic tridomain system \eqref{pbscale_gap}-\eqref{cond_ini_vws_gap} using specific approximation as follows (recall that our system is degenerate)
\begin{equation}
\begin{aligned}
&(\e+\delta_{n})\int_{\Gamma_{\e}^{1}} \pt_t \ov{u}_{i,n}^{1} \psi_{i}^{1} \ d\sigma_x-\e\int_{\Gamma_{\e}^{1}} \pt_t \ov{u}_{e,n}^{1} \psi_{i}^{1} \ d\sigma_x +\delta_{n}\int_{\Omega_{i,\e}^{1}} \pt_t u_{i,n}^{1} \phi_{i}^{1} \ dx
\\& \quad +(\dfrac{\e}{2}+\delta_{n})\int_{\Gamma_{\e}^{1,2}} \pt_t \ovo{u}_{i,n}^{1} \Psi_{i}^{1} \ d\sigma_x-\dfrac{\e}{2}\int_{\Gamma_{\e}^{1,2}} \pt_t \ovo{u}_{i,n}^{2} \Psi_{i}^{1} \ d\sigma_x 
\\& =\int_{\Gamma_{\e}^{1}} \e\left(-\I_{ion}(v_{n}^{1},w_{n}^{1})+\I_{app,\e}^{1}\right)\psi_{i}^{1} \ d\sigma_x
\\&\quad -\frac{1}{2}\int_{\Gamma_{\e}^{1,2}} \e \I_{gap}(s_{n})\Psi_{i}^{1} \ d\sigma_x-\int_{\Omega_{i,\e}^{1}}\mathrm{M}_{i}^{\e}\nabla u_{i,n}^{1}\cdot\nabla\phi_{i}^{1} \ dx
\end{aligned}
\label{Fv_i1_n_gap}
\end{equation}
\begin{equation}
\begin{aligned}
&(\e+\delta_{n})\int_{\Gamma_{\e}^{2}} \pt_t \ov{u}_{i,n}^{2} \psi_{i}^{2} \ d\sigma_x-\e\int_{\Gamma_{\e}^{2}} \pt_t \ov{u}_{e,n}^{2} \psi_{i}^{2} \ d\sigma_x +\delta_{n}\int_{\Omega_{i,\e}^{2}} \pt_t u_{i,n}^{2} \phi_{i}^{2} \ dx
\\&\quad -\dfrac{\e}{2}\int_{\Gamma_{\e}^{1,2}} \pt_t \ovo{u}_{i,n}^{1} \Psi_{i}^{2} \ d\sigma_x+(\dfrac{\e}{2}+\delta_{n})\int_{\Gamma_{\e}^{1,2}} \pt_t \ovo{u}_{i,n}^{2} \Psi_{i}^{2} \ d\sigma_x 
\\& =\int_{\Gamma_{\e}^{2}} \e\left(-\I_{ion}(v_{n}^{2},w_{n}^{2})+\I_{app,\e}^{2}\right)\psi_{i}^{2} \ d\sigma_x
\\&\quad +\frac{1}{2}\int_{\Gamma_{\e}^{1,2}} \e \I_{gap}(s_{n})\Psi_{i}^{2} \ d\sigma_x-\int_{\Omega_{i,\e}^{2}}\mathrm{M}_{i}^{\e}\nabla u_{i,n}^{2}\cdot\nabla\phi_{i}^{2} \ dx
\end{aligned}
\label{Fv_i2_n_gap}
\end{equation}
\begin{equation}
\begin{aligned}
&-\e\sum \limits_{k=1,2}\int_{\Gamma_{\e}^{k}} \pt_t \ov{u}_{i,n}^{k} \psi_{e}^{k} \ d\sigma_x+(\e+\delta_{n})\sum \limits_{k=1,2}\int_{\Gamma_{\e}^{k}} \pt_t \ov{u}_{e,n}^{k} \psi_{e}^{k} \ d\sigma_x +\delta_{n}\int_{\Omega_{e,\e}} \pt_t u_{e,n} \phi_{e} \ dx 
\\& =\sum \limits_{k=1,2}\int_{\Gamma_{\e}^{k}} \e\left(\I_{ion}(v_{n}^{k},w_{n}^{k})-\I_{app,\e}^{k}\right)\psi_{e}^{k} \ d\sigma_x-\int_{\Omega_{e,\e}}\mathrm{M}_{e}^{\e}\nabla u_{e,n}\cdot\nabla\phi_{e} \ dx
\end{aligned}
\label{Fv_e_n_gap}
\end{equation}
\begin{equation}
\int_{\Gamma_{\e}^{k}} \pt_t w_{n}^{k}\zeta^{k} \ d\sigma_x=\int_{\Gamma_{\e}^{k}} H\left(v_{n}^{k},w_{n}^{k}\right) \zeta^{k} \ d\sigma_x,
\label{Fv_d_n_gap}
\end{equation}
where the regularization parameter $\delta_{n}=\dfrac{1}{n},$ $\Theta_{i}^{k}={}^{t}\left(
\phi_{i}^{k} \ \ \psi_{i}^{k} \ \ \Psi_{i}^{k}
\right)\in \P_{i,n}^{k},$ $\zeta^{k}\in \mathcal{K}_{n}^{k}$ for $k=1,2,$  and $\Theta_{e}={}^{t}\left(
\phi_{e} \ \ \psi_{e}^{1} \ \ \psi_{e}^{2} \right)\in \P_{e,n}.$ The regularization terms multiplied by $\delta_{n}$ have been added to overcome degeneracy in \eqref{Fv_ike_ini_gap}.
Moreover, the resulting regularized problem is supplemented with initial conditions:
\begin{equation}
\begin{aligned}
& u_{i,n}^{k}(0,x)=u_{0,i,n}^{k}(x):=\sum \limits_{\ell=1}^{n} d_{i,\ell}^{k}(0)
\phi_{i,\ell}^{k}(x),
\\& \ov{u}_{i,n}^{k}(0,x)=\ov{u}_{0,i,n}^{k}(x):=\sum \limits_{\ell=1}^{n} d_{i,\ell}^{k}(0)\psi_{i,\ell}^{k}(x),
\\& \ovo{u}_{i,n}^{k}(0,x)=\ovo{u}_{0,i,n}^{k}(x):=\sum \limits_{\ell=1}^{n} d_{i,\ell}^{k}(0)\Psi_{i,\ell}^{k}(x), \quad d_{i,\ell}^{k}(0):=\langle U_{0,i}^{k},\Theta_{i,\ell}^{k}\rangle_{\V_{0,i}^{k}},
\\& u_{e,n} (0,x)=u_{0,e,n}(x):=\sum \limits_{\ell=1}^{n} d_{e,\ell}(0) \phi_{e,\ell},
\\& \ov{u}_{e,n}^{k}(0,x)=\ov{u}_{0,e,n}(x):=\sum \limits_{\ell=1}^{n} d_{e,\ell}(0)\psi_{e,\ell}^{k}(x), \quad d_{e,\ell}(0):=\langle U_{0,e},\Theta_{e,\ell}\rangle_{\V_{0,e}},
\\ & w_{n}^{k}(0,x)=w_{0,n}^{k}:=\sum \limits_{\ell=1}^{n} c_{\ell}^{k}(0)\zeta_{\ell}^{k}(x), \quad c_{\ell}^{k}(0):=\langle w_{0}^{k},\zeta_{\ell}^{k}\rangle_{L^{2}(\Gamma_{\e}^{k})},
\end{aligned}
\label{Approx_inisol}
\end{equation}
where $U_{0,i}^{k}:=U_{i}^{k}(0,x),$ for $k=1,2$ and $U_{0,e}:=U_{e}(0,x).$

Next, we prove in the following lemma the local existence of solutions for the previous regularized problem:

 \begin{lem}[Local existence of solutions for the regularized problems]  Assume that the conditions \eqref{A_Mie_gap}-\eqref{A_vw0_gap} hold. Then, there exits a positive time $0<t_0 \leq T$ such that System \eqref{Fv_i1_n_gap}-\eqref{Approx_inisol} admit a unique solution over the time interval $[0,t_0].$
 \label{local_exist}
\end{lem}
\begin{proof} The goal is to determine the coefficients $\mathbf{d}_{i}^{k}=\lbrace d_{i,\ell}^{k}\rbrace_{\ell=1}^{n},$ $\mathbf{d}_e=\lbrace d_{e,\ell}\rbrace_{\ell=1}^{n}$ and $\mathbf{c}^{k}=\lbrace c_{\ell}^{k}\rbrace_{\ell=1}^{n}$ for $k=1,2.$ For this purpose, if $n$ fixed, we choose $\Theta_{i}^{k}=\Theta_{i,m}^{k},$ $\Theta_{e}=\Theta_{e,m}$ and $\zeta^{k}=\zeta_{m}^{k}$ for $1\leq m \leq n$ and substitute the approximate solutions \eqref{Approx_sol} into \eqref{Fv_i1_n_gap}-\eqref{Fv_d_n_gap}. Then, the problem \eqref{Fv_i1_n_gap}-\eqref{Fv_d_n_gap}  is equivalent to the system of ordinary differential equations (ODE) in the following compact form:
\begin{equation}
\begin{aligned}
(\e+\delta_{n})\ov{\mathbb{A}}_{ii}^{1}(\mathbf{d}_{i}^{1})'-\e\ov{\mathbb{A}}_{ie}^{1}\mathbf{d}'_e+\delta_{n}\mathbb{A}_{ii}^{1}(\mathbf{d}_{i}^{1})'+(\dfrac{\e}{2}+\delta_{n})\ovo{\mathbb{A}}_{ii}^{1}(\mathbf{d}_{i}^{1})'-\dfrac{\e}{2}\ovo{\mathbb{A}}_{ii}^{1,2}(\mathbf{d}_{i}^{2})' &=& \mathbb{F}_{i}^{1}(t,\mathbf{d}_{i}^{1},\mathbf{d}_{i}^{2},\mathbf{d}_{e},\mathbf{c}^{1},\mathbf{c}^{2})
\\  (\e+\delta_{n})\ov{\mathbb{A}}_{ii}^{2}(\mathbf{d}_{i}^{2})'-\e\ov{\mathbb{A}}_{ie}^{2}\mathbf{d}'_e+\delta_{n}\mathbb{A}_{ii}^{2}(\mathbf{d}_{i}^{2})'-\dfrac{\e}{2}\ovo{\mathbb{A}}_{ii}^{1,2}(\mathbf{d}_{i}^{1})'+(\dfrac{\e}{2}+\delta_{n})\ovo{\mathbb{A}}_{ii}^{2}(\mathbf{d}_{i}^{2})' &=&\mathbb{F}_{i}^{2}(t,\mathbf{d}_{i}^{1},\mathbf{d}_{i}^{2},\mathbf{d}_{e},\mathbf{c}^{1},\mathbf{c}^{2})
\\  \sum \limits_{k=1,2}\left[ -\e \ov{\mathbb{A}}_{ie}^{k}(\mathbf{d}_{i}^{k})'+(\e+\delta_{n}) \ov{\mathbb{A}}_{ee}^{k}\mathbf{d}'_{e}\right] +\delta_{n}\mathbb{A}_{ee}\mathbf{d}_{e}' &=& \mathbb{F}_{e}(t,\mathbf{d}_{i}^{1},\mathbf{d}_{i}^{2},\mathbf{d}_{e},\mathbf{c}^{1},\mathbf{c}^{2})
\\  \mathbb{G}^{k}(\mathbf{c}^{k})' &=& \mathbb{H}^{k}(t,\mathbf{d}_{i}^{1},\mathbf{d}_{i}^{2},\mathbf{d}_{e},\mathbf{c}^{1},\mathbf{c}^{2})
\end{aligned}
\label{ODE_sys}
\end{equation}
with the $(\ell,m)$ entry of matrix:
\begin{itemize}
\item $\mathbb{A}_{ii}^{k}$ is $\langle\phi_{i,\ell}^{k}, \phi_{i,m}^{k}\rangle_{L^{2}(\Omega_{i,\e}^{k})}$ $\Big($ resp. of $\mathbb{A}_{ee}$ is $\langle\phi_{e,\ell}, \phi_{e,m}\rangle_{L^{2}(\Omega_{e,\e})}\Big)$,
\item $\ov{\mathbb{A}}_{ii}^{k}$ is $\langle\psi_{i,\ell}^{k}, \psi_{i,m}^{k}\rangle_{L^{2}(\Gamma_{\e}^{k})}$ $\Big($ resp.  of $\ov{\mathbb{A}}_{ee}^{k}$ is $\langle\psi_{e,\ell}^{k}, \psi_{e,m}^{k}\rangle_{L^{2}(\Gamma_{\e}^{k})}\Big),$
\item $\ov{\mathbb{A}}_{ie}^{k}$ is $\langle\psi_{i,\ell}^{k}, \psi_{e,m}^{k}\rangle_{L^{2}(\Gamma_{\e}^{k})},$
\item $\ovo{\mathbb{A}}_{ii}^{k}$ is $\langle\Psi_{i,\ell}^{k}, \Psi_{i,m}^{k}\rangle_{L^{2}(\Gamma_{\e}^{k})}$ $\Big($ resp. of $\ovo{\mathbb{A}}_{ii}^{1,2}$ is $\langle\Psi_{i,\ell}^{1}, \Psi_{i,m}^{2}\rangle_{L^{2}(\Gamma_{\e}^{1,2})}\Big),$
\item $\mathbb{G}^{k}$ is $\langle\zeta_{\ell}^{k}, \zeta_{m}^{k}\rangle_{L^{2}(\Gamma_{\e}^{k})},$
\end{itemize}
for $1\leq \ell,m\leq n$ and $k=1,2.$ Herein, the vectors $\mathbb{F}_{i}^{k},$ $\mathbb{F}_{e}$ and $\mathbb{H}^{k}$ for $k=1,2$ correspond to the right hand sides of the equations given in \eqref{Fv_i1_n_gap}-\eqref{Fv_d_n_gap}.

  Furthermore, the first three equations in ODE system \eqref{ODE_sys} can be written  as follows:
\begin{equation}
\mathbb{M}
\begin{bmatrix}
(\mathbf{d}_{i}^{1})' \\
(\mathbf{d}_{i}^{2})' \\
\mathbf{d}_{e}' \\
(\mathbf{c}^{1})' \\
(\mathbf{c}^{2})' \\
\end{bmatrix}
=\begin{bmatrix}
\mathbb{F}_{i}^{1} \\
\mathbb{F}_{i}^{2} \\
\mathbb{F}_{e} \\
\mathbb{H}^{1} \\
\mathbb{H}^{2} \\
\end{bmatrix}, 
\label{ODE_MDF_ie}
\end{equation} 
with $\mathbb{M}:=\mathbb{M}_1+\e\mathbb{M}_2$ and each matrix defined by:
 \begin{equation}
\mathbb{M}_1=
\begin{bmatrix}
\delta_n\left( \ov{\mathbb{A}}_{ii}^{1}+\mathbb{A}_{ii}^{1}+\ovo{\mathbb{A}}_{ii}^{1}\right)  & 0 & 0 & 0 & 0 \\
0  &\delta_n\left( \ov{\mathbb{A}}_{ii}^{2}+\mathbb{A}_{ii}^{2}+\ovo{\mathbb{A}}_{ii}^{2}\right)  & 0 & 0 & 0 \\
0 & 0 & \delta_n\left( \ov{\mathbb{A}}_{ee}^{1}+\ov{\mathbb{A}}_{ee}^{2}+\mathbb{A}_{ee}\right) & 0 & 0  \\
0 & 0 & 0 & \mathbb{G}^{1} & 0  \\
0 & 0 & 0 & 0 & \mathbb{G}^{2}  \\
\end{bmatrix}
\label{Matrix_1}
\end{equation}
and 
\begin{equation}
\mathbb{M}_2=
\begin{bmatrix}
\ov{\mathbb{A}}_{ii}^{1}+\frac{1}{2}\ovo{\mathbb{A}}_{ii}^{1} & -\frac{1}{2}\ovo{\mathbb{A}}_{ii}^{1,2}  & -\ov{\mathbb{A}}_{ie}^{1} & 0 & 0
\\
-\frac{1}{2}{}^{t}\ovo{\mathbb{A}}_{ii}^{1,2}  &\ov{\mathbb{A}}_{ii}^{2}+\frac{1}{2}\ovo{\mathbb{A}}_{ii}^{2} & -\ov{\mathbb{A}}_{ie}^{2} & 0 & 0
\\
-{}^{t}\ov{\mathbb{A}}_{ie}^{1} & -{}^{t}\ov{\mathbb{A}}_{ie}^{2} & \ov{\mathbb{A}}_{ee}^{1}+\ov{\mathbb{A}}_{ee}^{2} & 0 & 0
\\
 0 & 0 & 0 & 0 & 0
\\
 0 & 0 & 0 & 0 & 0
\\
\end{bmatrix}.
\label{Matrix_2}
\end{equation}
In order to write 
\begin{equation*}
\begin{bmatrix}
(\mathbf{d}_{i}^{1})' \\
(\mathbf{d}_{i}^{2})' \\
\mathbf{d}_{e}' \\
(\mathbf{c}^{1})' \\
(\mathbf{c}^{2})' \\
\end{bmatrix}
=\mathbb{M}^{-1}\begin{bmatrix}
\mathbb{F}_{i}^{1} \\
\mathbb{F}_{i}^{2} \\
\mathbb{F}_{e} \\
\mathbb{H}^{1} \\
\mathbb{H}^{2} \\
\end{bmatrix},
\end{equation*}
one needs to prove that the matrix $\mathbb{M}$ is invertible. According to Lemma \ref{lem_Matrix}, given below, the matrix $\mathbb{M}$ is symmetric positive definite, hence invertible. Consequently, we can write the ODE system \eqref{ODE_sys} in the form $z'(t)=F(t,z(t)).$ Finally, we prove the existence of a local solution $[0,t_0)$ to this ODE system with $t_0\in (0,T)$ (independent of the initial data). To this end, we show that th entries of $\mathbb{F}_{i}^{k},\mathbb{F}_{e}$ and $\mathbb{H}^{k}$ for $k=1,2$ are Caratheodory functions bounded by $L^1$ functions using the assumptions \eqref{A_Mie_gap}-\eqref{A_vw0_gap} by following the same strategy in \cite{bend}.
\end{proof}

\begin{lem}
For all $n\in \mathbb{N}^{\ast},$ the matrix $\mathbb{M}$ is positive definite.
\label{lem_Matrix}
\end{lem}
\begin{proof} Since we have $\mathbb{M}=\mathbb{M}_1+\e\mathbb{M}_2$ with $\mathbb{M}_1$ and $\mathbb{M}_2$ defined respectively by \eqref{Matrix_1}-\eqref{Matrix_2}. Note that by the orthonormality of the basis, the matrices $\ov{\mathbb{A}}_{ii}^{k},\mathbb{A}_{ii}^{k},\ovo{\mathbb{A}}_{ii}^{k},$ $\ov{\mathbb{A}}_{ee}^{k},$ $\mathbb{A}_{ee}$ and $\mathbb{G}^{k}$ are equal to the identity matrix $\mathbb{I}_{n\times n}$ for $k=1,2.$ So, the matrix
\begin{equation*}
\mathbb{M}_1=
\begin{bmatrix}
3\delta_n\mathbb{I}_{n\times n} & 0 & 0 & 0 & 0 \\
0  & 3\delta_n\mathbb{I}_{n\times n} & 0 & 0  \\
0 & 0 & 3\delta_n\mathbb{I}_{n\times n} & 0 & 0 \\
0 & 0 & 0 & \mathbb{I}_{n\times n} & 0 \\
0 & 0 & 0 & 0 & \mathbb{I}_{n\times n} \\
\end{bmatrix}
\end{equation*}
It suffices to show that the matrix $\mathbb{M}_2$ is positive semi-definite. Let $\mathbf{d}={}^{t}\left( \mathbf{d}_{i}^{1} \ \ \mathbf{d}_{i}^{2} \ \ \mathbf{d}_{e} \ \ \mathbf{c}^{1} \ \ \mathbf{c}^{2} \right)$ where $\mathbf{d}_{i}^{k}={}^{t}\left(\mathbf{d}_{i,1}^{k},\dots,\mathbf{d}_{i,n}^{k}\right)\in \R^{n},$  $\mathbf{d}_{e}={}^{t}\left(\mathbf{d}_{e,1},\dots,\mathbf{d}_{e,n}\right)\in \R^{n}$ and $\mathbf{c}^{k}={}^{t}\left(\mathbf{c}_{1}^{k},\dots,\mathbf{c}_{n}^{k}\right)\in \R^{n}$ for $k=1,2,$ we prove that ${}^{t}\mathbf{d}\mathbb{M}_2 \mathbf{d}\geq 0.$\\
Indeed, we have:
\begin{align*}
{}^{t}\mathbf{d}\mathbb{M}_2 \mathbf{d} &={}^{t}\mathbf{d}_{i}^{1}\left( \ov{\mathbb{A}}_{ii}^{1}+\frac{1}{2}\ovo{\mathbb{A}}_{ii}^{1}\right) \mathbf{d}_{i}^{1}+{}^{t}\mathbf{d}_{i}^{2}\left(\ov{\mathbb{A}}_{ii}^{2}+\frac{1}{2}\ovo{\mathbb{A}}_{ii}^{2}\right)\mathbf{d}_{i}^{2}+ {\,}^{t}\mathbf{d}_{e}
\left(\ov{\mathbb{A}}_{ee}^{1}+\ov{\mathbb{A}}_{ee}^{2}\right) \mathbf{d}_{e} 
\\& \quad -{\, }^{t}\mathbf{d}_{i}^{1}\ovo{\mathbb{A}}_{ii}^{1,2} \mathbf{d}_{i}^{2}-2{\, }^{t}\mathbf{d}_{i}^{1}\ov{\mathbb{A}}_{ie}^{1} \mathbf{d}_{e}-2{\, }^{t}\mathbf{d}_{i}^{2}\ov{\mathbb{A}}_{ie}^{2} \mathbf{d}_{e}
\\ &={}^{t}\mathbf{d}_{i}^{1}\ov{\mathbb{A}}_{ii}^{1}\mathbf{d}_{i}^{1}-2{\, }^{t}\mathbf{d}_{i}^{1}\ov{\mathbb{A}}_{ie}^{1} \mathbf{d}_{e}+ {\,}^{t}\mathbf{d}_{e}\ov{\mathbb{A}}_{ee}^{1} \mathbf{d}_{e}
\\& \quad +{}^{t}\mathbf{d}_{i}^{2}\ov{\mathbb{A}}_{ii}^{2}\mathbf{d}_{i}^{2}-2{\, }^{t}\mathbf{d}_{i}^{2}\ov{\mathbb{A}}_{ie}^{2} \mathbf{d}_{e}+ {\,}^{t}\mathbf{d}_{e}\ov{\mathbb{A}}_{ee}^{2} \mathbf{d}_{e}
\\& \quad +\frac{1}{2}{}^{t}\mathbf{d}_{i}^{1}\ovo{\mathbb{A}}_{ii}^{1}\mathbf{d}_{i}^{1}-{\, }^{t}\mathbf{d}_{i}^{1}\ovo{\mathbb{A}}_{ii}^{1,2} \mathbf{d}_{i}^{2}+\frac{1}{2}{}^{t}\mathbf{d}_{i}^{2}\ovo{\mathbb{A}}_{ii}^{2}\mathbf{d}_{i}^{2}
\\& :=\mathbb{E}_1+\mathbb{E}_2+\mathbb{E}_3
\end{align*}
We complete by showing that $\mathbb{E}_1\geq 0$ and the proof of the other terms $\mathbb{E}_2, \mathbb{E}_3$ is similar. Due the form of matrices and the orthonormality  of basis, we obtain:
\begin{align*}
\mathbb{E}_1 &={}^{t}\mathbf{d}_{i}^{1}\ov{\mathbb{A}}_{ii}^{1}\mathbf{d}_{i}^{1}-2{\, }^{t}\mathbf{d}_{i}^{1}\ov{\mathbb{A}}_{ie}^{1} \mathbf{d}_{e}+{\,}^{t}\mathbf{d}_{e}\ov{\mathbb{A}}_{ee}^{1} \mathbf{d}_{e}
\\& = \sum \limits_{\ell,m=1}^{n}\left[\mathbf{d}_{i,\ell}^{1}\mathbf{d}_{i,m}^{1}\displaystyle\int_{\Gamma_{\e}^{1}}\psi_{i,\ell}^{1}\psi_{i,m}^{1}-2\mathbf{d}_{i,\ell}^{1}\mathbf{d}_{e,m}\displaystyle\int_{\Gamma_{\e}^{1}}\psi_{i,\ell}^{1}\psi_{e,m}^{1}+\mathbf{d}_{e,\ell}\mathbf{d}_{e,m}\displaystyle\int_{\Gamma_{\e}^{1}}\psi_{e,\ell}^{1}\psi_{e,m}^{1}\right]\ d\sigma_x
\\& =\displaystyle\int_{\Gamma_{\e}^{1}} \left[\sum \limits_{\ell}^{n}\mathbf{d}_{i,\ell}^{1}\psi_{i,\ell}^{1}-\mathbf{d}_{e,\ell}\psi_{e,\ell}^{1}\right]^{2} \ d\sigma_x \ \geq 0.
\end{align*}
\end{proof}

\begin{rem} The above proof of the matrix $\mathbb{M}$ points out the role of the regularization term $\mathbb{M}_1$. It allows to obtain a matrix $\mathbb{M}$ in \eqref{ODE_MDF_ie} which is nonsingular, so that the resulting system of ODE is non-degenerate.
\end{rem}

 To prove global existence of the Faedo-Galerkin solutions on $[0,T),$  we derive a priori estimates, independent of the regularization parameter $n,$ bounding $u_{i,n}^{k},u_{e,n},v_{n}^{k},w_{n}^{k}$ for $k=1,2$ and $s_{n}$ in the next step.
\subsection*{Step 3: Energy estimates} The Faedo-Galerkin solutions satisfy the following weak formulations:
\begin{equation}
\begin{aligned}
&(\e+\delta_{n})\int_{\Gamma_{\e}^{1}} \pt_t \ov{u}_{i,n}^{1} \psi_{i,n}^{1} \ d\sigma_x-\e\int_{\Gamma_{\e}^{1}} \pt_t \ov{u}_{e,n} \psi_{i,n}^{1} \ d\sigma_x +\delta_{n}\int_{\Omega_{i,\e}^{1}} \pt_t u_{i,n}^{1} \varphi_{i,n}^{1} \ dx
\\&\quad +(\frac{\e}{2}+\delta_{n})\int_{\Gamma_{\e}^{1,2}} \pt_t \ovo{u}_{i,n}^{1} \Psi_{i,n}^{1} \ d\sigma_x-\frac{\e}{2}\int_{\Gamma_{\e}^{1,2}} \pt_t \ovo{u}_{i,n}^{2} \Psi_{i,n}^{1} \ d\sigma_x 
\\& =\int_{\Gamma_{\e}^{1}} \e\left(-\I_{ion}(v_{n}^{1},w_{n}^{1})+\I_{app,\e}^{1}\right)\psi_{i,n}^{1} \ d\sigma_x
\\&\quad -\frac{1}{2}\int_{\Gamma_{\e}^{1,2}} \e \I_{gap}(s_{n})\Psi_{i,n}^{1} \ d\sigma_x-\int_{\Omega_{i,\e}^{1}}\mathrm{M}_{i}^{\e}\nabla u_{i,n}^{1}\cdot\nabla\varphi_{i,n}^{1} \ dx
\end{aligned}
\label{Fv_i1n_n_gap}
\end{equation}
\begin{equation}
\begin{aligned}
&(\e+\delta_{n})\int_{\Gamma_{\e}^{2}} \pt_t \ov{u}_{i,n}^{2} \psi_{i,n}^{2} \ d\sigma_x-\e\int_{\Gamma_{\e}^{2}} \pt_t \ov{u}_{e,n} \psi_{i,n}^{2} \ d\sigma_x +\delta_{n}\int_{\Omega_{i,\e}^{2}} \pt_t u_{i,n}^{2} \varphi_{i,n}^{2} \ dx
\\&\quad -\frac{\e}{2}\int_{\Gamma_{\e}^{1,2}} \pt_t \ovo{u}_{i,n}^{1} \Psi_{i,n}^{2} \ d\sigma_x+(\frac{\e}{2}+\delta_{n})\int_{\Gamma_{\e}^{1,2}} \pt_t \ovo{u}_{i,n}^{2} \Psi_{i,n}^{2} \ d\sigma_x 
\\& =\int_{\Gamma_{\e}^{2}} \e\left(-\I_{ion}(v_{n}^{2},w_{n}^{2})+\I_{app,\e}^{2}\right)\psi_{i,n}^{2} \ d\sigma_x
\\&\quad +\frac{1}{2}\int_{\Gamma_{\e}^{1,2}} \e \I_{gap}(s_{n})\Psi_{i,n}^{2} \ d\sigma_x-\int_{\Omega_{i,\e}^{2}}\mathrm{M}_{i}^{\e}\nabla u_{i,n}^{2}\cdot\nabla\varphi_{i,n}^{2} \ dx
\end{aligned}
\label{Fv_i2n_n_gap}
\end{equation}
\begin{equation}
\begin{aligned}
&-\e\sum \limits_{k=1,2}\int_{\Gamma_{\e}^{k}} \pt_t \ov{u}_{i,n}^{k} \psi_{e,n}^{k} \ d\sigma_x+(\e+\delta_{n})\sum \limits_{k=1,2}\int_{\Gamma_{\e}^{k}} \pt_t \ov{u}_{e,n}^{k} \psi_{e,n}^{k} \ d\sigma_x +\delta_{n}\int_{\Omega_{e,\e}} \pt_t u_{e,n} \varphi_{e,n} \ dx 
\\& =\sum \limits_{k=1,2}\int_{\Gamma_{\e}^{k}} \e\left(\I_{ion}(v_{n}^{k},w_{n}^{k})-\I_{app,\e}^{k}\right)\psi_{e,n}^{k} \ d\sigma_x-\int_{\Omega_{e,\e}}\mathrm{M}_{e}^{\e}\nabla u_{e,n}\cdot\nabla\varphi_{e,n} \ dx
\end{aligned}
\label{Fv_en_n_gap}
\end{equation}
\begin{equation}
\int_{\Gamma_{\e}^{k}} \pt_t w_{n}^{k}e_{n}^{k} \ d\sigma_x=\int_{\Gamma_{\e}^{k}} H\left(v_{n}^{k},w_{n}^{k}\right) e_{n}^{k} \ d\sigma_x,
\label{Fv_dn_n_gap}
\end{equation}
where
\begin{align*}
\varphi_{i,n}^{k}(t,x):=\sum \limits_{\ell=1}^{n} a_{i,\ell}^{k}(t)\phi_{i,\ell}^{k}(x), \ \varphi_{e,n}(t,x):=\sum \limits_{\ell=1}^{n}a_{e,\ell}(t)\phi_{e,\ell}(x), \ e_{n}^{k}(t,x):=\sum \limits_{\ell=1}^{n}b_{\ell}(t)\xi_{\ell}^{k}(x),
\end{align*} 
for some given (absolutely continuous) coefficients $a_{i,\ell}^{k}(t),a_{e,\ell}(t),b_{\ell}^{k}(t)$ with $\ell=1,\dots,n$ and $k=1,2.$ Moreover, we recall that $\psi_{i,n}^{k}$ (resp. $\psi_{e,n}^{k}$) is the trace of $\varphi_{i,n}^{k}$ (resp. of $\varphi_{e,n}$) on $\Gamma_{\e}^{k}$ and $\Psi_{i,n}^{k}$ is the trace of $\varphi_{i,n}^{k}$ on $\Gamma_{\e}^{1,2}$ for $k=1,2.$ 

 We find now the a priori estimates of the solution of approximate
problem \eqref{Fv_i1n_n_gap}-\eqref{Fv_dn_n_gap}. First,  we  sum the three equations \eqref{Fv_i1n_n_gap}-\eqref{Fv_en_n_gap} to obtain the following weak formulation:
\begin{equation}
\begin{aligned}
&\sum \limits_{k=1,2} \int_{\Gamma_{\e}^{k}}\e \pt_t v_{n}^{k} \psi_{n}^{k} \ d\sigma_x+ \sum \limits_{k=1,2}\int_{\Gamma_{\e}^{k}} \delta_{n}\pt_t \ov{u}_{i,n}^{k} \psi_{i,n}^{k} \ d\sigma_x+ \sum \limits_{k=1,2}\int_{\Gamma_{\e}^{k}} \delta_{n}\pt_t \ov{u}_{e,n}^{k} \psi_{e,n}^{k} \ d\sigma_x 
\\&\quad +\frac{1}{2}\int_{\Gamma_{\e}^{1,2}}\e \pt_t s_{n} \Psi_{n} \ d\sigma_x+\sum \limits_{k=1,2}\int_{\Gamma_{\e}^{1,2}}\delta_{n} \pt_t \ovo{u}_{i,n}^{k} \Psi_{i,n}^{k} \ d\sigma_x
\\& \quad +\sum \limits_{k=1,2}\int_{\Omega_{i,\e}^{k}} \delta_{n}\pt_t u_{i,n}^{k} \varphi_{i,n}^{k} \ dx+\int_{\Omega_{e,\e}} \delta_{n}\pt_t u_{e,n} \varphi_{e,n} \ dx
\\& \quad +\sum \limits_{k=1,2}\int_{\Omega_{i,\e}^{k}}\mathrm{M}_{i}^{\e}\nabla u_{i,n}^{k}\cdot\nabla\varphi_{i,n}^{k} \ dx+\int_{\Omega_{e,\e}}\mathrm{M}_{e}^{\e}\nabla u_{e,n}\cdot\nabla\varphi_{e,n} \ dx
\\&\quad +\sum \limits_{k=1,2}\int_{\Gamma_{\e}^{k}} \e\I_{ion}\left(v_{n}^{k},w_{n}^{k}\right)\psi_{n}^{k} \ d\sigma_x+\frac{1}{2}\int_{\Gamma_{\e}^{1,2}} \e\I_{gap}\left(s_{n}\right)\Psi_{n} \ d\sigma_x
\\&=\sum \limits_{k=1,2}\int_{\Gamma_{\e}^{k}} \e\I_{app,\e}^{k}\psi_{n}^{k} \ d\sigma_x,
\end{aligned}
\label{Fv_trid}
\end{equation}
\begin{equation}
\int_{\Gamma_{\e}^{k}} \pt_t w_{n}^{k}e_{n}^{k} \ d\sigma_x=\int_{\Gamma_{\e}^{k}} H\left(v_{n}^{k},w_{n}^{k}\right) e_{n}^{k} \ d\sigma_x,
\label{Fv_dyn_trid}
\end{equation}
where $\psi_{n}^{k}=\psi_{i,n}^{k}-\psi_{e,n}^{k}$ for $k=1,2$ and $\Psi_{n}=\Psi_{i,n}^{1}-\Psi_{i,n}^{2}.$

 Next, we substitute $\varphi_{i,n}^{k}=u_{i,n}^{k},$ $\varphi_{e,n}=u_{e,n}$ and $e_{n}^{k}=\e \alpha_4 w_{n}^{k},$ respectively, in \eqref{Fv_trid}-\eqref{Fv_dyn_trid} to get the following equality:
\begin{equation}
\begin{aligned}
& \dfrac{1}{2}\dfrac{d}{dt} \Bigg[\sum\limits_{k=1,2}  \int_{\Gamma_{\e}^{k}} \card{\sqrt{\e}v_{n}^{k}}^2  \ d\sigma_x+ \sum \limits_{k=1,2}\int_{\Gamma_{\e}^{k}} \card{\sqrt{\delta_{n}}\ov{u}_{i,n}^{k}}^2 \ d\sigma_x+ \sum \limits_{k=1,2}\int_{\Gamma_{\e}^{k}} \card{\sqrt{\delta_{n}}\ov{u}_{e,n}^{k}}^2 \ d\sigma_x  
\\& \quad +\frac{1}{2}\int_{\Gamma_{\e}^{1,2}} \card{\sqrt{\e} s_{n}}^2 \ d\sigma_x+\sum \limits_{k=1,2}\int_{\Gamma_{\e}^{1,2}} \card{\sqrt{\delta_{n}}\ovo{u}_{i,n}^{k}}^2 \ d\sigma_x 
\\& \quad +\sum \limits_{k=1,2}\int_{\Omega_{i,\e}^{k}} \card{\sqrt{\delta_{n}}u_{i,n}^{k}}^2 \ dx+\int_{\Omega_{e,\e}} \card{\sqrt{\delta_{n}} u_{e,n}}^2  \ dx\Bigg]
\\& \quad +\sum \limits_{k=1,2}\int_{\Omega_{i,\e}^{k}}\mathrm{M}_{i}^{\e}\nabla u_{i,n}^{k}\cdot\nabla u_{i,n}^{k} \ dx+\int_{\Omega_{e,\e}}\mathrm{M}_{e}^{\e}\nabla u_{e,n}\cdot\nabla u_{e,n} \ dx
\\& \quad +\sum \limits_{k=1,2}\int_{\Gamma_{\e}^{k}} \e\I_{ion}\left(v_{n}^{k},w_{n}^{k}\right)v_{n}^{k} \ d\sigma_x+\frac{1}{2}\int_{\Gamma_{\e}^{1,2}} \e\I_{gap}\left(s_{n}\right)s_{n} \ d\sigma_x
\\&=\sum \limits_{k=1,2}\int_{\Gamma_{\e}^{k}} \e\I_{app,\e}^{k}v_{n}^{k} \ d\sigma_x,
\end{aligned}
\label{sub_u}
\end{equation}
\begin{align}
\dfrac{\alpha_4}{2}\dfrac{d}{dt}\int_{\Gamma_{\e}^{k}}  \card{\sqrt{\e}w_{n}^{k}}^2  \ d\sigma_x=\int_{\Gamma_{\e}^{k}} \e\alpha_4 H\left(v_{n}^{k},w_{n}^{k}\right) w_{n}^{k} \ d\sigma_x, \text{ for } k=1,2.
\label{sub_w}
\end{align}

 Integrating \eqref{sub_u}-\eqref{sub_w}  over $(0,t)$ for $t \in (0,t_0]$ in each equation and then summing the resulting equations, we procure the following equality using the assumption \eqref{A_H_I_gap} on $\I_{ion}$:
 \begin{equation}
 \begin{aligned}
& \dfrac{1}{2} \Bigg[\sum\limits_{k=1,2} \norm{\sqrt{\e}v_{n}^{k}}_{L^2(\Gamma_{\e}^{k})}^2+\alpha_4\sum\limits_{k=1,2} \norm{\sqrt{\e}w_{n}^{k}}_{L^2(\Gamma_{\e}^{k})}^2+\frac{1}{2}\norm{\sqrt{\e} s_{n}}_{L^2(\Gamma_{\e}^{1,2})}^2
\\&\quad + \sum \limits_{k=1,2}\norm{\sqrt{\delta_{n}}\ov{u}_{i,n}^{k}}_{L^2(\Gamma_{\e}^{k})}^2+ \sum \limits_{k=1,2}\norm{\sqrt{\delta_{n}}\ov{u}_{e,n}^{k}}_{L^2(\Gamma_{\e}^{k})}^2+\sum \limits_{k=1,2}\norm{\sqrt{\delta_{n}}\ovo{u}_{i,n}^{k}}_{L^2(\Gamma_{\e}^{1,2})}^2 
\\&\quad +\sum \limits_{k=1,2}\norm{\sqrt{\delta_{n}}u_{i,n}^{k}}_{L^2(\Omega_{i,\e}^{k})}^2+\norm{\sqrt{\delta_{n}} u_{e,n}}_{L^2(\Omega_{e,\e})}^2\Bigg]
\\& \quad +\sum \limits_{k=1,2}\int_{0}^{t}\int_{\Omega_{i,\e}^{k}}\mathrm{M}_{i}^{\e}\nabla u_{i,n}^{k}\cdot\nabla u_{i,n}^{k} \ dx d\tau+\int_{0}^{t}\int_{\Omega_{e,\e}}\mathrm{M}_{e}^{\e}\nabla u_{e,n}\cdot\nabla u_{e,n} \ dx d\tau
\\&\quad +\sum \limits_{k=1,2}\int_{0}^{t}\int_{\Gamma_{\e}^{k}} \e \tilde{\mathrm{I}}_{a,ion}\left(v_{n}^{k} \right)v_{n}^{k} \ d\sigma_xd\tau
\\&=\dfrac{1}{2} \Bigg[\sum\limits_{k=1,2} \norm{\sqrt{\e}v_{0,n}^{k}}_{L^2(\Gamma_{\e}^{k})}^2+\alpha_4\sum\limits_{k=1,2} \norm{\sqrt{\e}w_{0,n}^{k}}_{L^2(\Gamma_{\e}^{k})}^2+\frac{1}{2}\norm{\sqrt{\e} s_{0,n}}_{L^2(\Gamma_{\e}^{1,2})}^2
\\& \quad + \sum \limits_{k=1,2}\norm{\sqrt{\delta_{n}}\ov{u}_{0,i,n}^{k}}_{L^2(\Gamma_{\e}^{k})}^2+ \sum \limits_{k=1,2}\norm{\sqrt{\delta_{n}}\ov{u}_{0,e,n}^{k}}_{L^2(\Gamma_{\e}^{k})}^2+\sum \limits_{k=1,2}\norm{\sqrt{\delta_{n}}\ovo{u}_{0,i,n}^{k}}_{L^2(\Gamma_{\e}^{1,2})}^2 
\\& \quad +\sum \limits_{k=1,2}\norm{\sqrt{\delta_{n}}u_{0,i,n}^{k}}_{L^2(\Omega_{i,\e}^{k})}^2+\norm{\sqrt{\delta_{n}} u_{0,e,n}}_{L^2(\Omega_{e,\e})}^2\Bigg]
\\& \quad -\frac{1}{2}\int_{0}^{t}\int_{\Gamma_{\e}^{1,2}} \e\I_{gap}\left(s_{n}\right)s_{n} \ d\sigma_xd\tau+\sum \limits_{k=1,2}\int_{0}^{t}\int_{\Gamma_{\e}^{k}} \e\left( -\mathrm{I}_{b,ion}\left(w_{n}^{k}\right)v_{n}^{k}+\alpha_4 H\left(v_{n}^{k},w_{n}^{k}\right) w_{n}^{k} \right) \ d\sigma_xd\tau
\\&\quad +\sum \limits_{k=1,2}\int_{0}^{t}\int_{\Gamma_{\e}^{k}} \e \left(\beta_1 v_{n}^{k}+\beta_2 \right)v_{n}^{k} \ d\sigma_xd\tau+\sum \limits_{k=1,2}\int_{0}^{t}\int_{\Gamma_{\e}^{k}} \e\I_{app,\e}^{k}v_{n}^{k} \ d\sigma_xd\tau.
\end{aligned}
\label{sub_uw}
\end{equation}
We denote by $E_\ell$ with $\ell = 1,\dots, 9$ the terms of the previous equation which is rewritten as follows (to respect the order):
 \begin{equation*}
E_1+E_2+E_3+E_4=E_5+E_6+E_7+E_8+E_9.
 \end{equation*} 
Now, we estimate $E_\ell$ for $\ell=2,\dots,9$ as follows:
\begin{itemize}
\item Due the uniform ellipticity \eqref{A_Mie_gap} of $\mathrm{M}_{j}^{\e}$ for $j=i,e,$ we have
\begin{align*}
E_2+E_3 \geq \alpha \left(\sum \limits_{k=1,2}\int_{0}^{t}\norm{\nabla u_{i,n}^{k}}_{L^2(\Omega_{i,\e}^{k})}^2 \ d\tau+\int_{0}^{t}\norm{\nabla u_{e,n}}_{L^2(\Omega_{e,\e})}^2 \ d\tau \right)\geq 0.
\end{align*}
\item Using the assumption \eqref{A_tildeI_a_2} on $\tilde{\mathrm{I}}_{a,ion},$ we deduce that $E_4\geq 0.$
\item By the assumptions \eqref{A_vw0_gap} on the initial data, we have 
$E_5\leq C$ for some constant independent of $n$ and $\e.$
\item By the structure form of $\I_{gap}$ defined  in \eqref{gap_model}, we obtain
\begin{align*}
E_6\leq G_{gap} \int_{0}^{t} \norm{\sqrt{\e}s_{n}}_{L^2(\Gamma_{\e}^{1,2})}^2 \ d\tau.
\end{align*}
\item Using the assumption  on $\mathrm{I}_{b,ion}$ and $H$ defined as \eqref{A_H_Ib_a}, then we obtain
\begin{align*}
E_7\leq \alpha_5 \sum\limits_{k=1,2}\int_{0}^{t} \norm{\sqrt{\e}w_{n}^{k}}_{L^2(\Gamma_{\e}^{k})}^2 \ d\tau.
\end{align*}
\item It easy to estimate $E_8$ as follows
\begin{align*}
E_8\leq C \sum\limits_{k=1,2}\int_{0}^{t} \norm{\sqrt{\e}v_{n}^{k}}_{L^2(\Gamma_{\e}^{k})}^2 \ d\tau,
\end{align*}
with $C$ is constant independent of $n$ and $\e.$
\item By Young's inequality with the uniform $L^2$ boundedness \eqref{A_iapp_gap} of $\I_{app,\e}^{k}$, there exist constants $C_1,C_2>0$ independent of $n$ and $\e$ such that
\begin{align*}
E_9\leq C_1+C_2 \sum\limits_{k=1,2}\int_{0}^{t} \norm{\sqrt{\e}v_{n}^{k}}_{L^2(\Gamma_{\e}^{k})}^2 \ d\tau.
\end{align*}
\end{itemize} 

 Collecting all the estimates stated above, one obtains from \eqref{sub_uw} the following inequality for all $t\leq t_0,$
 \begin{equation}
\begin{aligned}
&\sum\limits_{k=1,2} \norm{\sqrt{\e}v_{n}^{k}}_{L^2(\Gamma_{\e}^{k})}^2+\sum\limits_{k=1,2} \norm{\sqrt{\e}w_{n}^{k}}_{L^2(\Gamma_{\e}^{k})}^2+\norm{\sqrt{\e} s_{n}}_{L^2(\Gamma_{\e}^{1,2})}^2
\\& \leq C\left(1+ \sum\limits_{k=1,2}\int_{0}^{t_0} \norm{\sqrt{\e}v_{n}^{k}}_{L^2(\Gamma_{\e}^{k})}^2 \ d\tau+\sum\limits_{k=1,2}\int_{0}^{t_0} \norm{\sqrt{\e}w_{n}^{k}}_{L^2(\Gamma_{\e}^{k})}^2 \ d\tau+ \int_{0}^{t_0} \norm{\sqrt{\e}s_{n}}_{L^2(\Gamma_{\e}^{1,2})}^2 \ d\tau\right).
\end{aligned}
\label{E_t_T}
 \end{equation}
By an application of Gronwall’s lemma in the last inequality, one gets
\begin{align*}
\sum\limits_{k=1,2} \norm{\sqrt{\e}v_{n}^{k}}_{L^2(\Gamma_{\e}^{k})}^2+\sum\limits_{k=1,2} \norm{\sqrt{\e}w_{n}^{k}}_{L^2(\Gamma_{\e}^{k})}^2+\norm{\sqrt{\e} s_{n}}_{L^2(\Gamma_{\e}^{1,2})}^2 \leq C.
\end{align*}
Hence, we conclude that
$$\sum\limits_{k=1,2} \norm{\sqrt{\e}v_{n}^{k}}_{L^{\infty}\left(0,T; L^2(\Gamma_{\e}^{k})\right)}^2+\sum\limits_{k=1,2} \norm{\sqrt{\e}w_{n}^{k}}_{L^{\infty}\left(0,T;L^2(\Gamma_{\e}^{k})\right)}^2+\norm{\sqrt{\e} s_{n}}_{L^{\infty}\left(0,T;L^2(\Gamma_{\e}^{1,2})\right)}^2 \leq C.$$ 
Then, we can deduce from this inequality that our approximate weak solution of the microscopic tridomain problem is global on $(0,T).$
 
 Moreover, one can obtain by exploiting this last inequality along with \eqref{sub_uw} the following a priori estimates for some constant $C>0$ not depending on $n$ and $\e$:
 \begin{equation}
 \begin{aligned}
&\sum\limits_{k=1,2} \norm{\sqrt{\e}v_{n}^{k}}_{L^{\infty}\left(0,T;L^2(\Gamma_{\e}^{k})\right)}^2+\sum\limits_{k=1,2} \norm{\sqrt{\e}w_{n}^{k}}_{L^{\infty}\left(0,T;L^2(\Gamma_{\e}^{k})\right)}^2+\norm{\sqrt{\e} s_{n}}_{L^{\infty}\left(0,T;L^2(\Gamma_{\e}^{1,2})\right)}^2
\\& \qquad + \sum \limits_{k=1,2}\norm{\sqrt{\delta_{n}}\ov{u}_{i,n}^{k}}_{L^{\infty}\left(0,T;L^2(\Gamma_{\e}^{1,2})\right)}^2+ \sum \limits_{k=1,2}\norm{\sqrt{\delta_{n}}\ov{u}_{e,n}^{k}}_{L^{\infty}\left(0,T;L^2(\Gamma_{\e}^{k})\right)}^2
\\&\qquad\qquad+\sum \limits_{k=1,2}\norm{\sqrt{\delta_{n}}\ovo{u}_{i,n}^{k}}_{L^{\infty}\left(0,T;L^2(\Gamma_{\e}^{k})\right)}^2\leq C,
 \end{aligned}
 \label{E_vwsinf}
  \end{equation}
  \begin{equation}
 \begin{aligned}
\sum \limits_{k=1,2}\norm{\nabla u_{i,n}^{k}}_{L^2(\Omega_{i,\e,T}^{k})}^2+\norm{\nabla u_{e,n}}_{L^2(\Omega_{e,\e,T})}^2 \leq C,
\end{aligned}
\label{E_uie}
  \end{equation}
  \begin{equation}
 \begin{aligned}
 \sum \limits_{k=1,2}\norm{\e \tilde{\mathrm{I}}_{a,ion}\left(v_{n}^{k} \right)v_{n}^{k}}_{L^1(\Gamma_{\e,T}^{k})} \leq C.
\end{aligned}
\label{E_tildIa}
  \end{equation}
    \begin{equation}
 \begin{aligned}
 \sum\limits_{k=1,2} \norm{\sqrt{\e}v_{n}^{k}}_{L^2(\Gamma_{\e,T}^{k})}^2+\sum\limits_{k=1,2} \norm{\sqrt{\e}w_{n}^{k}}_{L^2(\Gamma_{\e,T}^{k})}^2+\norm{\sqrt{\e} s_{n}}_{L^2(\Gamma_{\e,T}^{1,2})}^2\leq C,
 \end{aligned}
  \label{E_vwsL2}
  \end{equation}
 for some constant $C>0$ not depending on $n$ and $\e.$ 
 
  Furthermore, we deduce from \eqref{E_tildIa} together with assumption \eqref{A_tildeI_a_2} on $\tilde{\mathrm{I}}_{a,ion}$ the following estimation:
\begin{equation}
 \begin{aligned}
\sum\limits_{k=1,2} \norm{\e^{1/r}v_{n}^{k}}_{L^r(\Gamma_{\e,T}^{k})}^r\leq C, 
 \end{aligned}
 \label{E_vn_r}
  \end{equation}
 for some constant $C>0$ not depending on $n$ and $\e.$ The second estimate \eqref{E_vr_gap} in Theorem \ref{thm_micro_gap} is a direct consequence of \eqref{E_vn_r} and assumption \eqref{A_I_ab} on $\mathrm{I}_{a,ion}.$

  It remains to estimate on the $L^2$ norms of the intracellular and extracellular potentials which are need to complete the proof of Estimate \eqref{E_u_gap} on $H^1$. To do this end, we will use the next lemma, which is a consequence of the uniform Poincaré-Wirtinger's inequality and the trace theorem for $\e$-periodic surfaces.
\begin{lem} Let $u_{i}^{k} \in H^{1}\left( \Omega_{i,\e}^{k}\right)$ for $k=1,2$ and $u_{e} \in H^{1}\left( \Omega_{e,\e}\right).$ Set $v^{k}:=\left(u_{i}^{k}-u_{e}\right)\vert_{\Gamma_{\e}^{k}}$ for $k=1,2$.  Assume that the condition \eqref{normalization} holds, then there exists a positive constants $C$, independent of $\e$, such that  
\begin{equation}
\norm{u_{i}^{k}}_{L^2(\Omega_{i,\e}^{k})}^2\leq C\left(\norm{\sqrt{\e} v^{k}}_{L^2(\Gamma_{\e}^{k})}^2+\norm{\nabla u_{i}^{k}}_{L^2(\Omega_{i,\e}^{k})}^2+\norm{\nabla u_{e}}_{L^2(\Omega_{e,\e})}^2 \right), \text{ with } k=1,2.
\label{E_uik_vuike}
\end{equation}
\label{Poincare_trace_uik}
\end{lem}
\begin{proof}
We follow the same idea to the proof of Lemma 3.7 in \cite{karlsen}. Due the normalization condition \eqref{normalization}, Poincaré-Wirtinger's inequality implies that
 \begin{equation}
 \norm{ u_{e,n}}_{L^2\left( \Omega_{e,\e}\right)}^2\leq C \norm{\nabla u_{e,n}}_{L^2\left( \Omega_{e,\e}\right)}^2, 
 \label{E_PW_ue}
 \end{equation}
 for some constant $C$ independent on $n$ and $\e.$ Note that in the sequel $C$ is a generic constant whose value can change from one line to another.\\To estimate on the $L^2$ norms of $u_{i,n}^{k}$ for $k=1,2,$ we write 
 $$u_{i,n}^{k}=\widehat{u}_{i,n}^{k}+\widetilde{u}_{i,n}^{k},$$ 
 where $\widetilde{u}_{i,n}^{k}:=\dfrac{1}{\abs{\Omega_{i,\e}^{k}}}\displaystyle\int_{\Omega_{i,\e}^{k}} u_{i,n}^{k}dx$ is constant in $\Omega_{i,\e}^{k}$ and $\widehat{u}_{i,n}^{k}:=u_{i,n}^{k}-\widetilde{u}_{i,n}^{k}$ has zero mean in $\Omega_{i,\e}^{k}.$ Clearly, we see that for $k=1,2$
 \begin{equation*}
 \norm{ u_{i,n}^{k}}_{L^2\left( \Omega_{i,\e}^{k}\right)}^2= \norm{\widehat{u}_{i,n}^{k}}_{L^2\left( \Omega_{i,\e}^{k}\right)}^2+\norm{\widetilde{u}_{i,n}^{k}}_{L^2\left( \Omega_{i,\e}^{k}\right)}^2.
\end{equation*} 
In view of Poincaré-Wirtinger's inequality, one has
\begin{equation}
\norm{\widehat{u}_{i,n}^{k}}_{L^2\left( \Omega_{i,\e}^{k}\right)}^2\leq C \norm{\nabla\widehat{u}_{i,n}^{k}}_{L^2\left( \Omega_{i,\e}^{k}\right)}^2=C \norm{\nabla u_{i,n}^{k}}_{L^2\left( \Omega_{i,\e}^{k}\right)}^2 \text{ for } k=1,2.
\label{E_widehat_uik}
\end{equation}
Let us bound now $\norm{\widetilde{u}_{i,n}^{k}}_{L^2\left( \Omega_{i,\e}^{k}\right)}^2=\dfrac{\abs{\Omega_{i,\e}^{k}}}{\abs{\Gamma_{\e}^{k}}} \norm{\widetilde{u}_{i,n}^{k}}_{L^2\left(\Gamma_{\e}^{k}\right)}^2$ for $k=1,2.$ Since $\abs{\Gamma_{\e}^{k}}=\e^{-1}\abs{\Gamma^{k}}$ and $\abs{\Omega_{i,\e}^{k}}\leq \abs{\Omega},$ we deduce that
\begin{equation*}
\norm{\widetilde{u}_{i,n}^{k}}_{L^2\left( \Omega_{i,\e}^{k}\right)}^2\leq C \e \norm{\widetilde{u}_{i,n}^{k}}_{L^2\left(\Gamma_{\e}^{k}\right)}^2, \text{ for } k=1,2.
\end{equation*} 
It easy to check that
\begin{equation*}
\abs{\widetilde{u}_{i,n}^{k}}^{2}\leq C \left( \abs{u_{i,n}^{k}-u_{e,n}}^{2}+\abs{\widehat{u}_{i,n}^{k}}^{2}+\abs{u_{e,n}}^{2} \right).
\end{equation*}
Finally, we obtain for $k=1,2$
\begin{align*}
\norm{\widetilde{u}_{i,n}^{k}}_{L^2\left( \Omega_{i,\e}^{k}\right)}^2 & \leq C\left(  \e \norm{v_{n}^{k}}_{L^2\left(\Gamma_{\e}^{k}\right)}^2+ \e \norm{\widehat{u}_{i,n}^{k}}_{L^2\left(\Gamma_{\e}^{k}\right)}^2+\e \norm{u_{e,n}}_{L^2\left(\Gamma_{\e}^{k}\right)}^2\right) 
\\ & \leq C\e \norm{v_{n}^{k}}_{L^2\left(\Gamma_{\e}^{k}\right)}^2
\\ & \quad + C \left( \norm{\widehat{u}_{i,n}^{k}}_{L^2\left(\Omega_{i,\e}^{k}\right)}^2+\e^{2} \norm{\nabla \widehat{u}_{i,n}^{k}}_{L^2\left(\Omega_{i,\e}^{k}\right)}^2 \right) 
\\ & \quad +C \left( \norm{u_{e,n}}_{L^2\left(\Omega_{e,\e}\right)}^2+\e^{2}\norm{\nabla u_{e,n}}_{L^2\left(\Omega_{e,\e}\right)}^2\right)
\\ & \leq C\left(\norm{\sqrt{\e} v_{n}^{k}}_{L^2(\Gamma_{\e}^{k})}^2+\norm{\nabla u_{i,n}^{k}}_{L^2(\Omega_{i,\e}^{k})}^2+\norm{\nabla u_{e,n}}_{L^2(\Omega_{e,\e})}^2 \right), 
\end{align*}
where the second inequality is a direct consequence of the trace theorem and the final one is a result of \eqref{E_PW_ue} and \eqref{E_widehat_uik}. This completes the proof of this lemma. 
\end{proof}
Now, Estimate \eqref{E_uie} and \eqref{E_uik_vuike} imply that
 \begin{equation}
 \norm{ u_{e,n}}_{L^2\left( 0,T;H^1(\Omega_{e,\e})\right)}\leq C, 
 \label{E_ue_H1}
 \end{equation}
 for some constant $C$ independent on $n$ and $\e.$
Furthermore, we have $\norm{\sqrt{\e} v_{n}^k}_{L^2(\Gamma_{\e,T}^{k})}^2\leq C$ for $k=1,2$. Then Estimates \eqref{E_uik_vuike}, \eqref{E_uie} and \eqref{E_ue_H1} ensure that for $k=1,2,$ 
\begin{equation}
\norm{ u_{i,n}^{k}}_{L^2\left( 0,T;H^1(\Omega_{i,\e}^{k})\right)}\leq C.
\label{E_ui_H1}
\end{equation}
This completes the proof of \eqref{E_vw_gap}-\eqref{E_vr_gap} in Theorem \ref{thm_micro_gap}.

  Now we turn to find some uniform estimates on the time derivatives by following \cite{bend} which will be useful for the passage to the limit. We notice first for $k=1,2$ that,
  \begin{align*}
&\iint_{\Omega_{i,\e,T}^{k}}\mathrm{M}_{i}^{\e}\nabla u_{i,n}^{k}\cdot\nabla \left( \pt_t u_{i,n}^{k}\right) \ dx=\dfrac{1}{2} \int_{0}^{T}\pt_t\left( \int_{\Omega_{i,\e}^{k}}\mathrm{M}_{i}^{\e}\nabla u_{i,n}^{k}\cdot\nabla u_{i,n}^{k} \ dx\right)dt 
\\&= \dfrac{1}{2} \Bigg[\int_{\Omega_{i,\e}^{k}}\mathrm{M}_{i}^{\e}\nabla u_{i,n}^{k}(T,\cdot)\cdot\nabla u_{i,n}^{k}(T,\cdot) \ dx-\int_{\Omega_{i,\e}^{k}}\mathrm{M}_{i}^{\e}\nabla u_{i,n}^{k}(0,\cdot)\cdot\nabla u_{i,n}^{k}(0,\cdot) \ dx\Bigg],
  \end{align*}
  and
  \begin{align*}
&\iint_{\Gamma_{\e,T}^{k}} \mathrm{I}_{a,ion}\left(v_{n}^{k}\right)\pt_t v_{n}^{k}  \ d\sigma_x dt=\int_{0}^{T}\pt_t\left( \int_{\Gamma_{\e}^{k}}\int_{0}^{v_{n}^{k}} \mathrm{I}_{a,ion}\left(\tilde{v}_{n}^{k}\right) \ d\tilde{v}_{n}^{k}d\sigma_x\right) dt
\\&=\int_{\Gamma_{\e}^{k}} \int_{0}^{v_{n}^{k}(T,\cdot)} \mathrm{I}_{a,ion}\left(v_{n}^{k} \right) \ dv_{n}^{k}d\sigma_x-\int_{\Gamma_{\e}^{k}} \int_{0}^{v_{n}^{k}(0,\cdot)}\mathrm{I}_{a,ion}\left(v_{n}^{k}\right) \ dv_{n}^{k}d\sigma_x.
  \end{align*}

   Next, we substitute $\varphi_{i,n}^{k}=\pt_t u_{i,n}^{k},$ $\varphi_{e,n}=\pt_t u_{e,n}$ and $e_{n}^{k}=\e \alpha_4 \pt_t w_{n}^{k},$ respectively, in \eqref{Fv_trid}-\eqref{Fv_dyn_trid} then integrate in time to deduce using the previous equalities:
 \begin{equation}
 \begin{aligned}
& \sum\limits_{k=1,2} \norm{\sqrt{\e}\pt_t v_{n}^{k}}_{L^2(\Gamma_{\e,T}^{k})}^2+\alpha_4\sum\limits_{k=1,2} \norm{\sqrt{\e}\pt_t w_{n}^{k}}_{L^2(\Gamma_{\e,T}^{k})}^2+\frac{1}{2}\norm{\sqrt{\e} \pt_t s_{n}}_{L^2(\Gamma_{\e,T}^{1,2})}^2
\\& \quad + \sum \limits_{k=1,2}\norm{\sqrt{\delta_{n}}\pt_t \ov{u}_{i,n}^{k}}_{L^2(\Gamma_{\e,T}^{k})}^2+ \sum \limits_{k=1,2}\norm{\sqrt{\delta_{n}}\pt_t\ov{u}_{e,n}^{k}}_{L^2(\Gamma_{\e,T}^{k})}^2+\sum \limits_{k=1,2}\norm{\sqrt{\delta_{n}}\pt_t\ovo{u}_{i,n}^{k}}_{L^2(\Gamma_{\e,T}^{1,2})}^2 
\\&\quad +\sum \limits_{k=1,2}\norm{\sqrt{\delta_{n}}\pt_tu_{i,n}^{k}}_{L^2(\Omega_{i,\e,T}^{k})}^2+\norm{\sqrt{\delta_{n}} \pt_t u_{e,n}}_{L^2(\Omega_{e,\e,T})}^2
\\& \quad+\dfrac{1}{2} \Bigg[\sum \limits_{k=1,2}\int_{\Omega_{i,\e}^{k}}\mathrm{M}_{i}^{\e}\nabla u_{i,n}^{k}\cdot\nabla u_{i,n}^{k}(T,\cdot) \ dx+\int_{\Omega_{e,\e}}\mathrm{M}_{e}^{\e}\nabla u_{e,n}(T,\cdot)\cdot\nabla u_{e,n}(T,\cdot) \ dx
\\&\quad +\sum \limits_{k=1,2}\int_{\Gamma_{\e}^{k}} \int_{0}^{v_{n}^{k}(T,\cdot)}\e \tilde{\mathrm{I}}_{a,ion}\left(v_{n}^{k}\right) \ dv_{n}^{k}d\sigma_x\Bigg]
\\&=\dfrac{1}{2} \Bigg[\sum \limits_{k=1,2}\int_{\Omega_{i,\e}^{k}}\mathrm{M}_{i}^{\e}\nabla u_{i,n}^{k}(0,\cdot)\cdot\nabla u_{i,n}^{k}(0,\cdot) \ dx+\int_{\Omega_{e,\e}}\mathrm{M}_{e}^{\e}\nabla u_{e,n}(0,\cdot)\cdot\nabla u_{e,n}(0,\cdot) \ dx
\\& \quad +\sum \limits_{k=1,2}\int_{\Gamma_{\e}^{k}} \int_{0}^{v_{n}^{k}(0,\cdot)}\e \mathrm{I}_{a,ion}\left(v_{n}^{k}\right) \ dv_{n}^{k}d\sigma_x + \sum \limits_{k=1,2}\int_{\Gamma_{\e}^{k}}\int_{0}^{v_{n}^{k}(T,\cdot)}\e\left(\beta_1 v_{n}^{k}+\beta_2\right)   \ dv_{n}^{k}d\sigma_x\Bigg]
\\& \quad -\frac{1}{2}\iint_{\Gamma_{\e,T}^{1,2}} \e\I_{gap}\left(s_{n}\right)\pt_t s_{n} \ d\sigma_xd\tau+\sum \limits_{k=1,2}\int_{\Gamma_{\e,T}^{k}} \e\left( -\mathrm{I}_{b,ion}\left(w_{n}^{k}\right)\pt_t v_{n}^{k}+\alpha_4 H\left(v_{n}^{k},w_{n}^{k}\right) \pt_t w_{n}^{k} \right) \ d\sigma_xd\tau
\\& \quad+\sum \limits_{k=1,2}\int_{\Gamma_{\e,T}^{k}} \e\I_{app,\e}^{k}\pt_t v_{n}^{k} \ d\sigma_xd\tau.
\end{aligned}
\label{sub_dtuw}
\end{equation} 
We denote by $E'_\ell$ with $\ell = 1,\dots, 6$ the terms of the previous equation which is rewritten as follows (to respect the order):
 \begin{equation*}
E'_1+E'_2=E'_3+E'_4+E'_5+E'_6,
 \end{equation*} 
 where
 \begin{align*}
 E'_1 &:= \sum\limits_{k=1,2} \norm{\sqrt{\e}\pt_t v_{n}^{k}}_{L^2(\Gamma_{\e,T}^{k})}^2+\alpha_4\sum\limits_{k=1,2} \norm{\sqrt{\e}\pt_t w_{n}^{k}}_{L^2(\Gamma_{\e,T}^{k})}^2+\frac{1}{2}\norm{\sqrt{\e} \pt_t s_{n}}_{L^2(\Gamma_{\e,T}^{1,2})}^2
\\&+ \sum \limits_{k=1,2}\norm{\sqrt{\delta_{n}}\pt_t \ov{u}_{i,n}^{k}}_{L^2(\Gamma_{\e,T}^{k})}^2+ \sum \limits_{k=1,2}\norm{\sqrt{\delta_{n}}\pt_t\ov{u}_{e,n}^{k}}_{L^2(\Gamma_{\e,T}^{k})}^2+\sum \limits_{k=1,2}\norm{\sqrt{\delta_{n}}\pt_t\ovo{u}_{i,n}^{k}}_{L^2(\Gamma_{\e,T}^{1,2})}^2 
\\&+\sum \limits_{k=1,2}\norm{\sqrt{\delta_{n}}\pt_tu_{i,n}^{k}}_{L^2(\Omega_{i,\e,T}^{k})}^2+\norm{\sqrt{\delta_{n}} \pt_t u_{e,n}}_{L^2(\Omega_{e,\e,T})}^2.
 \end{align*}
 
Now, we estimate $E'_\ell$ for $\ell=2,\dots,6$ as follows:
\begin{itemize}
\item Due the uniform ellipticity \eqref{A_Mie_gap} of $\mathrm{M}_{j}^{\e}$ for $j=i,e,$ with the monotonicity \eqref{A_tildeI_a_2} on $\tilde{\mathrm{I}}_{a,ion}$, then we have 
 \begin{align*}
E'_2 \geq & \  \alpha \left(\sum \limits_{k=1,2}\norm{\nabla u_{i,n}^{k}(T,\cdot)}_{L^2(\Omega_{i,\e}^{k})}^2+\norm{\nabla u_{e,n}(T,\cdot)}_{L^2(\Omega_{e,\e})}^2 \right)
\\&+\sum \limits_{k=1,2}\int_{\Gamma_{\e}^{k}} \int_{0}^{v_{n}^{k}(T,\cdot)}\e \tilde{\mathrm{I}}_{a,ion}\left(v_{n}^{k}(T,\cdot) \right) \ d\sigma_xdv_{n}^{k}\geq 0.
\end{align*}
\item Furthermore, using the a priori estimate \eqref{E_vwsinf} with the assumption on $\mathrm{I}_{a,ion}$ and on the initial data, one gets
\begin{align*}
E'_3 \leq & \  \beta \left(\sum \limits_{k=1,2}\norm{\nabla u_{i,n}^{k}(0,\cdot)}_{L^2(\Omega_{i,\e}^{k})}^2+\norm{\nabla u_{e,n}(0,\cdot)}_{L^2(\Omega_{e,\e})}^2 \right)
\\&+\alpha_1\sum \limits_{k=1,2}\int_{\Gamma_{\e}^{k}}\e\left( \card{v_{n}^{k}(0,\cdot)}^r+\card{v_{n}^{k}(0,\cdot)}\right)  \ d\sigma_x
\\&+\dfrac{\beta_1}{2}\sum \limits_{k=1,2}\int_{\Gamma_{\e}^{k}}\e \card{v_{n}^{k}(T,\cdot)}^2  \ d\sigma_x+\beta_2\sum \limits_{k=1,2}\int_{\Gamma_{\e}^{k}}\e \card{v_{n}^{k}(T,\cdot)}  \ d\sigma_x\leq C_3
\end{align*}
for some  constant $C_3$ independent of $n$ and $\e.$
\item By the structure form of $\I_{gap}$ defined  in \eqref{gap_model}, we obtain using Young's inequality with estimate \eqref{E_vwsL2}
\begin{align*}
E'_4&\leq \dfrac{G_{gap}}{2}\norm{\sqrt{\e}s_n}_{L^2(\Gamma_{\e,T}^{1,2})}^2 +\dfrac{1}{4}\norm{\sqrt{\e}\pt_t s_{n}}_{L^2(\Gamma_{\e,T}^{1,2})}^2.
\\& \leq C_4+\dfrac{1}{4}\norm{\sqrt{\e}\pt_t s_{n}}_{L^2(\Gamma_{\e,T}^{1,2})}^2
\end{align*}
with $C_4$ independent of $n$ and $\e.$
\item Similarly, using the assumption  on $\mathrm{I}_{b,ion}$ and $H$ defined as \eqref{A_I_ab}-\eqref{A_H_Ib_a}, then we obtain using Young's inequality with the estimate \eqref{E_vwsL2}
\begin{align*}
E'_5 \leq C_5+\dfrac{1}{2}\sum \limits_{k=1,2}\left( \norm{\sqrt{\e}\pt_t v_{n}^{k}}_{L^2(\Gamma_{\e,T}^{k})}^2+\norm{\sqrt{\e}\pt_t w_{n}^{k}}_{L^2(\Gamma_{\e}^{k})}^2\right) 
\end{align*}
with $C_5$ independent of $n$ and $\e.$
\item By Young's inequality with the uniform $L^2$ boundedness \eqref{A_iapp_gap} of $\I_{app,\e}^{k}$, there exist constants $C_1,C_2>0$ independent of $n$ and $\e$ such that
\begin{align*}
E'_6\leq C_6+\dfrac{1}{2} \sum\limits_{k=1,2} \norm{\sqrt{\e}\pt_t v_{n}^{k}}_{L^2(\Gamma_{\e,T}^{k})}^2,
\end{align*}
with $C_6$ independent of $n$ and $\e.$
\end{itemize} 
Exploiting all this estimates along with \eqref{sub_dtuw}, one obtains
 \begin{equation}
 \begin{aligned}
& \sum\limits_{k=1,2} \norm{\sqrt{\e}\pt_t v_{n}^{k}}_{L^2(\Gamma_{\e,T}^{k})}^2+\alpha_4\sum\limits_{k=1,2} \norm{\sqrt{\e}\pt_t w_{n}^{k}}_{L^2(\Gamma_{\e,T}^{k})}^2+\norm{\sqrt{\e} \pt_t s_{n}}_{L^2(\Gamma_{\e,T}^{1,2})}^2
\\&\quad + \sum \limits_{k=1,2}\norm{\sqrt{\delta_{n}}\pt_t \ov{u}_{i,n}^{k}}_{L^2(\Gamma_{\e,T}^{k})}^2+ \sum \limits_{k=1,2}\norm{\sqrt{\delta_{n}}\pt_t\ov{u}_{e,n}^{k}}_{L^2(\Gamma_{\e,T}^{k})}^2+\sum \limits_{k=1,2}\norm{\sqrt{\delta_{n}}\pt_t\ovo{u}_{i,n}^{k}}_{L^2(\Gamma_{\e,T}^{1,2})}^2 
\\&\quad +\sum \limits_{k=1,2}\norm{\sqrt{\delta_{n}}\pt_tu_{i,n}^{k}}_{L^2(\Omega_{i,\e,T}^{k})}^2+\norm{\sqrt{\delta_{n}} \pt_t u_{e,n}}_{L^2(\Omega_{e,\e,T})}^2 \leq C
\end{aligned}
\label{E_dtvws}
\end{equation}
for some constant $C>0$ not depending on $n$ and $\e.$ \\

 The next steps is devoted to completing the proof of Theorem \ref{thm_micro_gap} and to passing to the limit when $n$ goes  to infinity. Further, it treat the uniqueness of the weak solutions to System \eqref{pbscale_gap}-\eqref{cond_ini_vws_gap}
\subsection*{Step 4: Passage to the limit and global existence of solutions}
In view of \eqref{E_ue_H1}-\eqref{E_ui_H1}, we can see that $v_n^{k},$ $\ov{u}_{j,n}^{k}$ are bounded in $L^2\left(0,T;H^{1/2}(\Gamma_{\e}^{k}) \right)$ for $j=i,e$ and $k=1,2$ using the standard trace lemma. Similarly, it easy to check that $s_n$ and $\ovo{u}_{i,n}^{k}$ are bounded in $L^2\left(0,T;H^{1/2}(\Gamma_{\e}^{1,2}) \right)$ for $k=1,2.$
Furthermore, we deduce from \eqref{E_dtvws} the uniform bound on $\pt_t v_{n}^{k}$ in $L^2(\Gamma_{\e,T}^{k})$ for $k=1,2$ and the uniform bound on $\pt_t s_{n}$ in $L^2(\Gamma_{\e,T}^{1,2})$. Recall that by the Aubin-Lions compactness criterion, the following injection
\begin{equation*}
\mathcal{W}:=\left\lbrace  u\in L^2\left(0,T;H^{1/2}(\Gamma_{\e})\right) \text{ and } \pt_t u \in L^2\left(0,T;H^{-1/2}(\Gamma_{\e})\right)\right\rbrace \subset L^2(\Gamma_{\e,T})
\end{equation*}
is compact with $\Gamma_{\e}:=\Gamma_{\e}^{k},\Gamma_{\e}^{1,2}$ for $k=1,2.$ Hence, we can assume there exist limit functions $u_{i,\e}^{1},$ $u_{i,\e}^{2},$ $u_{e,\e},$ $v_{\e}^{1},$ $v_{\e}^{2},$ $s_{\e},$ $w_{\e}$ with $v_{\e}^{k}=\ov{u}_{i,\e}^{k}-\ov{u}_{e,\e}^{k}$ on $\Gamma_{\e,T}^{k}$ for $k=1,2$ and $s_{\e}=\ovo{u}_{i,\e}^{1}-\ovo{u}_{i,\e}^{2}$ on $\Gamma_{\e,T}^{1,2},$ such that as $n\rightarrow \infty$ ( for fixed $\e$ and up to an unlabeled subsequence)
\begin{equation}
\begin{cases}
v_{n}^{k}\rightarrow v_{\e}^{k} \text{ a.e. in } \Gamma_{\e}^{k}, \text{ strongly in } L^2(\Gamma_{\e,T}^{k}), 
\\ \text{ and weakly } L^2\left(0,T; H^{1/2}(\Gamma_{\e}^{k})\right) \text{ for } k=1,2, 
\\ s_{n}\rightarrow s_{\e} \text{ a.e. in } \Gamma_{\e}^{1,2}, \text{ strongly in } L^2(\Gamma_{\e,T}^{1,2}),
\\ \text{ and weakly } L^2\left(0,T; H^{1/2}(\Gamma_{\e}^{1,2})\right),
\\ w_{n}^{k}\rightharpoonup w_{\e}^{k} \text{ weakly in } L^2(\Gamma_{\e,T}^{k}),
\\ u_{i,n}^{k}\rightharpoonup u_{i,\e}^{k} \text{ weakly in } L^2\left(0,T;H^{1}(\Omega_{i,\e}^{k})\right) \text{ for } k=1,2,
\\ u_{e,n}\rightharpoonup u_{e,\e} \text{ weakly in } L^2\left(0,T;H^{1}(\Omega_{e,\e})\right),
\\ \mathrm{I}_{a,ion}\left(v_{n}^{k}\right)\rightarrow \mathrm{I}_{a,ion}\left(v_{\e}^{k}\right) \text{ a.e. in } \Gamma_{\e}^{k}, \text{ weakly in } L^{r/(r-1)}(\Gamma_{\e,T}^{k}),
\end{cases}
\label{limit_vws}
\end{equation}
and
\begin{equation}
\begin{cases}
 \pt_t v_{n}^{k}\rightharpoonup \pt_t v_{\e}^{k} \text{ weakly in } L^2(\Gamma_{\e,T}^{k}),
\\ \pt_t w_{n}^{k}\rightharpoonup \pt_t w_{\e}^{k} \text{ weakly in } L^2(\Gamma_{\e,T}^{k}) \text{ for } k=1,2,
\\ \pt_t s_{n}\rightharpoonup \pt_t s_{\e} \text{ weakly in } L^2(\Gamma_{\e,T}^{1,2}).
\end{cases}
\label{limit_dtvws}
\end{equation}
Moreover, using again estimate \eqref{E_dtvws}, we get for $j=i,e$ and $k=1,2,$
\begin{equation}
\begin{cases}
\sqrt{\delta_{n}} \pt_t \ov{u}_{j,\e}^{k} \rightharpoonup 0 \text{ in } D'\left(0,T;L^{2}(\Gamma_{\e}^{k})\right) \text{ for } j=i,e,
\\
\\ \sqrt{\delta_{n}} \pt_t \ovo{u}_{i,\e}^{k} \rightharpoonup 0 \text{ in } D'\left(0,T;L^{2}(\Gamma_{\e}^{1,2})\right),\  \sqrt{\delta_{n}} \pt_t u_{i,\e}^{k} \rightharpoonup 0 \text{ in } D'\left(0,T;L^{2}(\Omega_{i,\e}^{k})\right),
\\
\\  \text{ and } \sqrt{\delta_{n}}\pt_t u_{e,\e} \rightharpoonup 0 \text{ in } D'\left(0,T;L^{2}(\Omega_{e,\e})\right).
\end{cases}
\label{limit_deltnu}
\end{equation} 

The last difficulty is to prove that the nonlinear term $\mathrm{I}_{a,ion}\left(v_{n}^{k}\right)$ converges weakly to the term $\mathrm{I}_{a,ion}\left(v_{\e}^{k}\right)$ for $k=1,2.$ Since $v_{n}^{k}$ converges strongly to $v_{\e}^{k}$ in $L^2(\Gamma_{\e,T}^{k})$, we can extract a subsequence, such that $v_{n}^{k}$ converges almost everywhere to $v_{\e}^{k}$ in $\Gamma_{\e}^{k}$ for $k=1,2.$ Moreover, since $\mathrm{I}_{a,ion}$ is continuous, we have 
 \begin{equation}
\mathrm{I}_{a,ion}\left(v_{n}^{k}\right)\rightarrow \mathrm{I}_{a,ion}\left(v_{\e}^{k}\right) \text{ a.e. in } \Gamma_{\e}^{k} \text{ for } \ k=1,2.
\label{limit_Iaion_ae}
\end{equation}
However, using a classical result (see Lemma 1.3 in \cite{lions1969}):
\begin{equation}
\mathrm{I}_{a,ion}\left(v_{n}^{k}\right)\rightharpoonup \mathrm{I}_{a,ion}\left(v_{\e}^{k}\right), \text{ weakly in } L^{r/(r-1)}(\Gamma_{\e,T}^{k}) \text{ for } \ k=1,2.
\label{limit_Iaion_weakly}
\end{equation}

\begin{rem}
By our choice of basis, it is clear that $\ov{u}_{j,n}^{k}(0,x) \rightarrow \ov{u}_{0,j,\e}^{k}$ in $L^{2}(\Gamma_{\e}^{k})$ for $k=1,2$ and $j=i,e$. Furthermore, we have $\ovo{u}_{i,n}^{k}(0,x) \rightarrow \ovo{u}_{0,i,\e}^{k}$ in $L^{2}(\Gamma_{\e}^{k})$ for $k=1,2$.
\end{rem}
 
Keeping in mind \eqref{limit_vws}-\eqref{limit_Iaion_weakly}, we obtain by letting $n\rightarrow \infty$ in the weak formulation \eqref{Fv_trid}-\eqref{Fv_dyn_trid}
\begin{equation}
\begin{aligned}
&\sum \limits_{k=1,2} \int_{\Gamma_{\e}^{k}}\e \pt_t v_{\e}^{k} \psi^{k} \ d\sigma_x+\frac{1}{2}\int_{\Gamma_{\e}^{1,2}}\e \pt_t s_{\e} \Psi \ d\sigma_x
\\&\quad +\sum \limits_{k=1,2}\int_{\Omega_{i,\e}^{k}}\mathrm{M}_{i}^{\e}\nabla u_{i,\e}^{k}\cdot\nabla\varphi_{i}^{k} \ dx+\int_{\Omega_{e,\e}}\mathrm{M}_{e}^{\e}\nabla u_{e,\e}\cdot\nabla\varphi_{e} \ dx
\\&\quad +\sum \limits_{k=1,2}\int_{\Gamma_{\e}^{k}} \e\I_{ion}\left(v_{\e}^{k},w_{\e}^{k}\right)\psi^{k} \ d\sigma_x+\frac{1}{2}\int_{\Gamma_{\e}^{1,2}} \e\I_{gap}\left(s_{\e}\right)\Psi \ d\sigma_x
\\&=\sum \limits_{k=1,2}\int_{\Gamma_{\e}^{k}} \e\I_{app,\e}^{k}\psi^{k} \ d\sigma_x,
\end{aligned}
\label{Fv_trid_passage}
\end{equation}
\begin{equation}
\int_{\Gamma_{\e}^{k}} \pt_t w_{\e}^{k}e^{k} \ d\sigma_x=\int_{\Gamma_{\e}^{k}} H\left(v_{\e}^{k},w_{\e}^{k}\right) e^{k} \ d\sigma_x,
\label{Fv_dyn_trid_passage}
\end{equation}
for all $\varphi_{i}^{k}\in H^{1}(\Omega_{i,\e}^{k}),$ $\varphi_{e}\in H^{1}(\Omega_{e,\e})$ with $\psi^{k}=\psi_{i}^{k}-\psi_{e}^{k}\in H^{1/2}(\Gamma_{\e}^{k})\cap L^{r}(\Gamma_{\e}^{k})$ for $k=1,2,$  $\Psi=\Psi_{i}^{1}-\Psi_{i}^{2}\in  L^{2}(\Gamma_{\e}^{1,2})$ and $e^{k}\in L^{2}(\Gamma_{\e}^{k})$ for $k=1,2.$
  Finally, it only remains to be proved that $v_{\e}^{k},$ $w_{\e}^{k}$ for $k=1,2$ and $s_{\e}$ satisfy the initial conditions stated in Definition \ref{Fv_gap}. 
  Using the weak formulation \eqref{Fv_i1n_n_gap}-\eqref{Fv_en_n_gap}, we see that $v_{\e}^{k}(0,x)=v_{0,\e}^{k}(x)$ a.e. on $\Gamma_{\e,T}^{k},$ since, by construction, $\ov{u}_{j,n}^{k}(0,x) \rightarrow \ov{u}_{0,j,\e}^{k}$ in $L^{2}(\Gamma_{\e}^{k})$ for $k=1,2$ and $j=i,e$. The same argument holds for $w_{\e}^{k}$ for $k=1,2$ and $s_{\e}$.

\subsection*{Step 5: Uniqueness of solutions} This step prove that there there exists at most one weak solution of \eqref{Fv_trid_passage}-\eqref{Fv_dyn_trid_passage}. We assume that $u^{\ell}=\left(u_{i,\e}^{1,\ell}, u_{i,\e}^{2,\ell}, u_{e,\e}^{\ell}, w_{\e}^{1,\ell}, w_{\e}^{2,\ell} \right),$ $\ell\in\left\lbrace \ell',\ell''\right\rbrace $ are two  weak solutions in the sense of Definition \ref{Fv_gap} with same initial data. Thus, this weak formulations hold respectively for $u_{i,\e}^{k,\ell'}-u_{i,\e}^{k,\ell''}$ and $w_{\e}^{k,\ell'}-w_{\e}^{k,\ell''}$ for $k=1,2$. 

 Firstly, we substitute $\varphi_{i}^{k}= u_{i,\e}^{k,\ell'}-u_{i,\e}^{k,\ell''},$ $\varphi_{e}= u_{e,\e}^{\ell'}-u_{e,\e}^{\ell''},$ and $e^{k}= \e\left( w_{\e}^{k,\ell'}-w_{\e}^{k,\ell''}\right) ,$ $k=1,2,$ respectively in \eqref{Fv_trid_passage}-\eqref{Fv_dyn_trid_passage}. Then, we add the resulting equations and integrate over $(0,t)$ for $0<t \leq T$ to get 
\begin{equation*}
\begin{aligned}
&\dfrac{1}{2}\Bigg[\sum \limits_{k=1,2} \int_{\Gamma_{\e}^{k}}\left( \e \card{\left( v_{\e}^{k,\ell'}-v_{\e}^{k,\ell''}\right) (t,\cdot) }^2+\e\card{\left( w_{\e}^{k,\ell'}-w_{\e}^{k,\ell''}\right) (t,\cdot) }^2 \right) \ d\sigma_x
\\& \quad +\frac{1}{2}\int_{\Gamma_{\e}^{1,2}}\e \card{\left( s_{\e}^{\ell'}-s_{\e}^{\ell''}\right) (t,\cdot) }^2  \ d\sigma_x\Bigg]
\\& \quad +\sum \limits_{k=1,2}\iint_{\Omega_{i,\e,t}^{k}}\mathrm{M}_{i}^{\e}\nabla \left( u_{i,\e}^{k,\ell'}-u_{i,\e}^{k,\ell''}\right)\cdot\nabla\left( u_{i,\e}^{k,\ell'}-u_{i,\e}^{k,\ell''}\right) \ dxd\tau
\\& \quad +\iint_{\Omega_{e,\e,t}}\mathrm{M}_{e}^{\e}\nabla\left( u_{e,\e}^{\ell'}-u_{e,\e}^{\ell''}\right)\cdot\nabla\left( u_{e,\e}^{\ell'}-u_{e,\e}^{\ell''}\right) \ dxd\tau
\\& \quad +\sum \limits_{k=1,2}\iint_{\Gamma_{\e,t}^{k}} \e\left( \tilde{\mathrm{I}}_{a,ion}\left(v_{\e}^{k,\ell'}\right)- \tilde{\mathrm{I}}_{a,ion}\left(v_{\e}^{k,\ell''}\right)\right)\left( v_{\e}^{k,\ell'}-v_{\e}^{k,\ell''}\right) \ d\sigma_xd\tau
\\&=\dfrac{1}{2}\Bigg[\sum \limits_{k=1,2} \int_{\Gamma_{\e}^{k}}\left( \e \card{\left( v_{0,\e}^{k,\ell'}-v_{0,\e}^{k,\ell''}\right)}^2+\e\card{\left( w_{0,\e}^{k,\ell'}-w_{0,\e}^{k,\ell''}\right)}^2 \right) \ d\sigma_x
\\& \quad +\frac{1}{2}\int_{\Gamma_{\e}^{1,2}}\e \card{\left( s_{0,\e}^{\ell'}-s_{0,\e}^{\ell''}\right)}^2 \ d\sigma_x\Bigg]
\\& \quad +\beta_1\sum \limits_{k=1,2}\iint_{\Gamma_{\e,t}^{k}} \e\left( v_{\e}^{k,\ell'}-v_{\e}^{k,\ell''}\right)^2 \ d\sigma_xd\tau
\\& \quad -\sum \limits_{k=1,2}\iint_{\Gamma_{\e,t}^{k}} \e \mathrm{I}_{b,ion}\left(w_{\e}^{k,\ell'}-w_{\e}^{k,\ell''}\right)\left( v_{\e}^{k,\ell'}-v_{\e}^{k,\ell''}\right) \ d\sigma_xd\tau
\\& \quad -\sum \limits_{k=1,2}\iint_{\Gamma_{\e,t}^{k}} \e\left( H\left(v_{\e}^{k,\ell'},w_{\e}^{k,\ell'}\right)- H\left(v_{\e}^{k,\ell''},w_{\e}^{k,\ell''}\right)\right)\left( w_{\e}^{k,\ell'}-w_{\e}^{k,\ell''}\right) \ d\sigma_xd\tau
\\& \quad -\frac{1}{2}\iint_{\Gamma_{\e,t}^{1,2}} \e\I_{gap}\left( s_{\e}^{\ell'}-s_{\e}^{\ell''}\right)\left( s_{\e}^{\ell'}-s_{\e}^{\ell''}\right) \ d\sigma_xd\tau
\\& \quad + \sum \limits_{k=1,2}\iint_{\Gamma_{\e,t}^{k}} \e\I_{app,\e}^{k}\left( v_{\e}^{k,\ell'}-v_{\e}^{k,\ell''}\right) \ d\sigma_xd\tau.
\end{aligned}
\end{equation*}

 Due the uniform ellipticity \eqref{A_Mie_gap} of $\mathrm{M}_{j}^{\e}$ for $j=i,e,$ we have
 \begin{align*}
 &\sum \limits_{k=1,2}\iint_{\Omega_{i,\e,t}^{k}}\mathrm{M}_{i}^{\e}\nabla \left( u_{i,\e}^{k,\ell'}-u_{i,\e}^{k,\ell''}\right)\cdot\nabla\left( u_{i,\e}^{k,\ell'}-u_{i,\e}^{k,\ell''}\right) \ dxd\tau
\\&\quad +\iint_{\Omega_{e,\e,t}}\mathrm{M}_{e}^{\e}\nabla\left( u_{e,\e}^{\ell'}-u_{e,\e}^{\ell''}\right)\cdot\nabla\left( u_{e,\e}^{\ell'}-u_{e,\e}^{\ell''}\right) \ dxd\tau 
\\ & \geq \alpha \left(\sum \limits_{k=1,2}\norm{\nabla\left( u_{i,\e}^{k,\ell'}-u_{i,\e}^{k,\ell''}\right)}^{2}_{L^2(\Omega_{i,\e,t}^{k})} + \norm{\nabla\left( u_{e,\e}^{\ell'}-u_{e,\e}^{\ell''}\right)}^{2}_{L^2(\Omega_{e,\e,t})} \right)\geq 0.
\end{align*}  
Furthermore, thanks to the monotonicity assumption \eqref{A_tildeI_a_2} on $\tilde{\mathrm{I}}_{a,ion},$ we deduce that
\begin{align*}
&\sum \limits_{k=1,2}\iint_{\Gamma_{\e,t}^{k}} \e\left( \tilde{\mathrm{I}}_{a,ion}\left(v_{\e}^{k,\ell'}\right)- \tilde{\mathrm{I}}_{a,ion}\left(v_{\e}^{k,\ell''}\right)\right)\left( v_{\e}^{k,\ell'}-v_{\e}^{k,\ell''}\right) \ d\sigma_xd\tau  \geq 0.
\end{align*}
Moreover, by the linearity of $\mathrm{I}_{b,ion},$ $H$ and $\I_{gap},$ we can deduce using Young's inequality the following estimation
\begin{equation}
\begin{aligned}
&\dfrac{1}{2}\Bigg[\sum \limits_{k=1,2} \int_{\Gamma_{\e}^{k}}\left( \e \card{\left( v_{\e}^{k,\ell'}-v_{\e}^{k,\ell''}\right) (t,\cdot) }^2+\e\card{\left( w_{\e}^{k,\ell'}-w_{\e}^{k,\ell''}\right) (t,\cdot) }^2 \right) \ d\sigma_x+
\\&\quad +\frac{1}{2}\int_{\Gamma_{\e}^{1,2}}\e \card{\left( s_{\e}^{\ell'}-s_{\e}^{\ell''}\right) (t,\cdot) }^2  \ d\sigma_x\Bigg]
\\ & \quad +\alpha \left(\sum \limits_{k=1,2}\norm{\nabla\left( u_{i,\e}^{k,\ell'}-u_{i,\e}^{k,\ell''}\right)}^{2}_{L^2(\Omega_{i,\e,t}^{k})} + \norm{\nabla\left( u_{e,\e}^{\ell'}-u_{e,\e}^{\ell''}\right)}^{2}_{L^2(\Omega_{e,\e,t})} \right)
\\&  \leq \dfrac{1}{2}\Bigg[\sum \limits_{k=1,2} \int_{\Gamma_{\e}^{k}}\left( \e \card{\left( v_{0,\e}^{k,\ell'}-v_{0,\e}^{k,\ell''}\right)}^2+\e\card{\left( w_{0,\e}^{k,\ell'}-w_{0,\e}^{k,\ell''}\right)}^2 \right) \ d\sigma_x
\\&\qquad \qquad \qquad +\frac{1}{2}\int_{\Gamma_{\e}^{1,2}}\e \card{\left( s_{0,\e}^{\ell'}-s_{0,\e}^{\ell''}\right)}^2 \ d\sigma_x\Bigg]
\\& \qquad +C\Bigg[\sum \limits_{k=1,2} \int_{0}^{t}\int_{\Gamma_{\e}^{k}}\left( \e \card{\left( v_{\e}^{k,\ell'}-v_{\e}^{k,\ell''}\right)}^2+\e\card{\left( w_{\e}^{k,\ell'}-w_{\e}^{k,\ell''}\right)}^2 \right) \ d\sigma_x
\\&\qquad \qquad \qquad +\frac{1}{2}\int_{\Gamma_{\e}^{1,2}}\e\card{\left( s_{\e}^{\ell'}-s_{\e}^{\ell''}\right)}^2  \ d\sigma_xd\tau\Bigg]
\end{aligned}
\label{E_v1v2_uniqueness}
\end{equation}
 where $C>0$ is a constant independent of $\e.$ Thus, we obtain by applying Gronwall's inequality 
 \begin{align*}
&\dfrac{1}{2}\Bigg[\sum \limits_{k=1,2} \int_{\Gamma_{\e}^{k}}\left( \e \card{\left( v_{\e}^{k,\ell'}-v_{\e}^{k,\ell''}\right) (t,\cdot) }^2+\e\card{\left( w_{\e}^{k,\ell'}-w_{\e}^{k,\ell''}\right) (t,\cdot) }^2 \right) \ d\sigma_x
\\& \quad +\frac{1}{2}\int_{\Gamma_{\e}^{1,2}}\e \card{\left( s_{\e}^{\ell'}-s_{\e}^{\ell''}\right) (t,\cdot) }^2  \ d\sigma_x\Bigg]
\\& \leq C\Bigg[\sum \limits_{k=1,2} \int_{\Gamma_{\e}^{k}}\left( \e \card{\left( v_{0,\e}^{k,\ell'}-v_{0,\e}^{k,\ell''}\right)}^2+\e\card{\left( w_{0,\e}^{k,\ell'}-w_{0,\e}^{k,\ell''}\right)}^2 \right) \ d\sigma_x
\\&\qquad \qquad \qquad +\frac{1}{2}\int_{\Gamma_{\e}^{1,2}}\e \card{\left( s_{0,\e}^{\ell'}-s_{0,\e}^{\ell''}\right)}^2 \ d\sigma_x\Bigg]
\end{align*}
for some constant $C>0.$ Hence, we deduce that $v_{\e}^{k,\ell'}=v_{\e}^{k,\ell''},$ $w_{\e}^{k,\ell'}=w_{\e}^{k,\ell''}$ for $k=1,2$ and $s_{\e}^{\ell'}=s_{\e}^{\ell''}.$
Moreover, using Estimation \eqref{E_v1v2_uniqueness}, we conclude that
\begin{align*}
& \nabla\left( u_{e,\e}^{\ell'}-u_{e,\e}^{\ell''}\right)=0 &\text{ a.e. on } \Omega_{e,\e,t}, 
\\& \nabla\left( u_{i,\e}^{k,\ell'}-u_{i,\e}^{k,\ell''}\right)=0 &\text{ a.e. on } \Omega_{i,\e,t}^{k},
\end{align*}
which means that $u_{e,\e}^{\ell'}=u_{e,\e}^{\ell''}+c$ and $u_{i,\e}^{k,\ell'}=u_{i,\e}^{k,\ell''}+c$ for $k=1,2$. On the one hand, due to the normalization condition \eqref{normalization}, $c=0$ and $u_{e,\e}^{\ell'}=u_{e,\e}^{\ell''}.$ 
 On the other hand, the estimation \eqref{E_uik_vuike} holds for $u_{i,\e}^{k,\ell'}-u_{i,\e}^{k,\ell''}$ which gives
\begin{align*}
\norm{u_{i,\e}^{k,\ell'}-u_{i,\e}^{k,\ell''}}_{L^2(\Omega_{i,\e}^{k})}^2 &\leq C\Bigg(\norm{\sqrt{\e} \left( v_{\e}^{k,\ell'}-v_{\e}^{k,\ell''}\right)}_{L^2(\Gamma_{\e}^{k})}^2+\norm{\nabla \left( u_{i,\e}^{k,\ell'}-u_{i,\e}^{k,\ell''} \right)}_{L^2(\Omega_{i,\e}^{k})}^2
\\ & \quad+\norm{\nabla \left( u_{e,\e}^{\ell'}-u_{e,\e}^{\ell''}\right)}_{L^2(\Omega_{e,\e})}^2 \Bigg), \text{ with } k=1,2.
\end{align*}
In addition, we have  $v_{\e}^{k,\ell'}=v_{\e}^{k,\ell''}$ so we obtain finally $u_{i,\e}^{k,\ell'}=u_{i,\e}^{k,\ell''}$ for $k=1,2.$ This gives the uniqueness proof of weak solutions.

\subsection*{Funding}
This research was supported by IEA-CNRS in the context of HIPHOP project.

\bibliographystyle{plain}
 \bibliography{Hom}

\end{document}